\documentclass[12pt, reqno]{amsart}

\usepackage{amssymb}
\usepackage{graphicx}
\usepackage[cmtip,all]{xy}

        \newcommand{\K}{\ensuremath{\Bbbk}}
        \newcommand{\N}{\ensuremath{\mathbb{N}}}
        \newcommand{\Q}{\ensuremath{\mathbb{Q}}}
        
        \newcommand{\Z}{\ensuremath{\mathbb{Z}}}

\newcommand{\qnum}[2]{[#1]_{#2}}
\DeclareMathOperator{\ann}{ann}
\DeclareMathOperator{\ch}{char}
\DeclareMathOperator{\GK}{GKdim}
\DeclareMathOperator{\gr}{gr}
\DeclareMathOperator{\spa}{-span}
        
	  \newcommand{\End}{\textup{End}}	
        \newcommand{\Hom}{\textup{Hom}}
        \newcommand{\grHom}{\textup{grHom}}

        \newcommand{\del}{\ensuremath{\partial}}

        \theoremstyle{plain}
                \newtheorem{theorem}{Theorem}[section]
                \newtheorem{lemma}[theorem]{Lemma}
                \newtheorem{corollary}[theorem]{Corollary}
                \newtheorem{proposition}[theorem]{Proposition}
        \theoremstyle{definition}
                \newtheorem{definition}[theorem]{Definition}
                \newtheorem{example}[theorem]{Example}
                \newtheorem{notation}[theorem]{Notation}
                \newtheorem{remark}[theorem]{Remark}
                
        \numberwithin{equation}{section}
        
        \newcommand{\ignore}[1]{}
        
        \newcommand{\mynote}[1]{}

\begin{document}

        \title[Noetherian Algebras of quantum differential operators]
        {Noetherian Algebras of quantum differential operators}



        \author[Iyer]{Uma N. Iyer$^1$}
        \author[Jordan]{David A. Jordan$^2$}
        \email{${}^1$uma.iyer@bcc.cuny.edu, ${}^2$d.a.jordan@sheffield.ac.uk}
        \address{${}^1$Department of Mathematics and Computer Science,
          2155 University Avenue, Bronx, New York 10453,USA.}
        \address{${}^2$School of Mathematics and Statistics,
                               University of Sheffield,
                               Hicks Building,
                               Sheffield,
                               S3 7RH.
                               UK.}

\date{\today}


\begin{abstract}
We consider algebras  of quantum differential operators, for appropriate bicharacters  on a polynomial algebra in one indeterminate and
for the coordinate algebra of quantum $n$-space for $n\geq 3$. In the former case a set of generators for the quantum differential operators was identified in work by the first author and T. C. McCune but it was not known whether the algebra is Noetherian. We answer this question affirmatively,
setting it in a more general context involving the behaviour Noetherian condition under localization at the powers of a single element. In the latter case we determine the algebra of quantum differential operators as a skew group algebra of the group $\Z^n$ over a quantized Weyl algebra. It follows from this description that this algebra is a simple right and left Noetherian domain.
\end{abstract}

\maketitle


\section{Introduction}
This paper brings together two approaches to algebras
of quantum differential operators. On the one hand there is the formal axiomatic approach
of  Lunts and Rosenberg \cite{LR} where a notion of quantum differential operator is defined and
one can then set out to identify the algebra of operators for familiar algebras and analyse their properties.
The definition of these quantum differential operators requires the base $\K$-algebra $R$ to be $\Gamma$-graded for some abelian group $\Gamma$ and depends on the choice of a bicharacter $\beta: \Gamma \times \Gamma \to \K^*$.
The first author and T. C. McCune have contributed to this theory through their identification of the quantum differential operators for appropriate bicharacters, of the polynomial algebra $\K[x]$ in one variable \cite{IM1} and the coordinate algebra of the quantum plane \cite{IM2}. A second approach is to study algebras that are either generated by particular operators of an intuitively quantum differential nature or are generalizations of familiar algebras of differential operators. Examples of the latter in the work of the second author are the quantized Weyl algebras \cite{J1}, arising from the quantum calculus of Maltsiniotis \cite{malt}, and the algebras generated by two or more Eulerian derivatives \cite{J2}. Quantized Weyl algebras are iterated skew polynomial algebras over the field $\K$, making many of their properties transparent.  Algebras constructed in this approach are often subalgebras of the more formal algebras of quantum differential operators of Lunts and Rosenberg and one can then ask how close the relationship is and whether known properties of the subalgebras can enhance the understanding of the full algebras. The most substantial two projects in the paper are instances of this.

The algebra $D$ of quantum differential operators on the polynomial algebra $\K [x]$, for a particular bicharacter depending on a parameter $q$,  was studied in \cite{IM1} where a set of generators was identified together with relations satisfied by those generators and some structural properties. The main question left open was whether $D$ is Noetherian and we here we shall answer this positively, under reasonable conditions on the parameter $q$. The algebra $D_x$ of quantum differential operators on the Laurent polynomial algebra $\K [x^{\pm 1}]$ is more accessible than $D$ and is obtained from $D$ by localizing at the powers of $x$ which, although $x$ is not normal in $D$, form a right and left Ore set. In particular, $D_x$ is readily seen to be right and left Noetherian. This raises an interesting general problem, namely, given a domain $A$ and an element $a\in A$ such that the powers of $a$  form a right and left Ore set and the localization $A_a$ is Noetherian,
to find conditions on $A$ and $a$ which guarantee that $A$ itself is Noetherian. We make a contribution to this problem by showing that under conditions that, although they appear strong, are satisfied by $D$ and $x$. This is heavily dependent on the work of Rogalski \cite{R}. In addition to the Noetherian result, we find a basis for $D$ and a complete set of defining relations. The approach here follows that used in \cite{J2} in the study of algebras generated by Eulerian derivatives.

Our second project concerns the identification of the algebra of quantum differential operators for the coordinate algebra of quantum $n$-space for $n\geq 3$ and a natural bicharacter. We shall see here that, under reasonable conditions, this algebra is a skew group algebra of the group $\Z^n$ over a quantized Weyl algebra. The approach here is first to identify quantum differential operators that generate a subalgebra that is isomorphic to such an algebra and then to show that this subalgebra is the full algebra of quantum differential operators. Many structural properties follow from this description. In particular the algebra of quantum differential operators is a simple right and left Noetherian domain.

Other algebras for which we consider quantum differential operators are quantum tori, the coordinate algebra of the quantum plane
and the quantum exterior algebra. Operators related to quantum differential operators,  such as derivations, skew derivations, and
generalized derivations, on the algebras that we consider have been well studied, see
\cite{A1,A2,A3,AC,OP}
and references therein.

The paper is arranged as follows.
The preliminaries of quantum differential operators are explained in Section \ref{Prelim}.
Section \ref{poly} is concerned with quantum differential operators on polynomial algebras and Laurent polynomial algebras. This section includes a summary of the results of \cite{IM1} and our discussion of the Noetherian problem discussed above.
In Section~\ref{qtorus}, we consider quantum differential operators over the coordinate algebras of quantum $n$-space and quantum tori. We
recall the results on the quantum plane from \cite{IM2} and
then study the case of the  coordinate algebra of quantum $n$-space, concentrating on the case $n\geq 3$ where there are distinct differences to the cases $n=1,2$. The short Section~\ref{ext} deals with the quantum experior algebra for which the algebra of quantum differential
operators is the entire algebra of linear maps on the quantum
exterior algebra.

\subsection*{Acknowledgements}

We thank T. C. McCune for useful discussions.
The first author gratefully acknowledges support from
the PSC-CUNY Grants, award \# 65413-00 43. She also thanks
University of Hyderabad, India, for excellent working conditions where
part of this work was undertaken, and S. Datt (see \ref{datt}) for useful
discussions.

\section{Preliminaries  on Quantum differential
operators (\cite{LR}) } \label{Prelim}
Throughout $\K$ will be a field of characteristic $0$ and $\K^*$ will denote the multiplicative group $\K\setminus \{0 \}$.
Let $\Gamma$ be an abelian group. Fix a bicharacter $\beta : \Gamma \times
\Gamma \longrightarrow \K^*$.
 Let $R$ be a $\Gamma$-graded $\K$-algebra and $M$ a $\Gamma$-graded
 $R$-bimodule. Let $h(R)$, respectively $h(M)$, denote the set of homogeneous elements of $R$,
respectively $M$.

 Let $\mathcal{Z}_{q}(M)$ denote the {\it quantum-center} of $M$ defined as
	the $\K\spa$ of those elements $m \in h(M)$ for which there
        exists $d\in \Gamma$ such that
	$	mr = \beta (d , d_r) rm$
		for all  $r\in h(R)$.
 For each $a\in \Gamma$, define $\sigma_a \in \grHom_{\K}(M,M)$ defined by
	$\sigma_a (m) = \beta (a,d_m)m$ for  $m\in h(M)$, and
        extend $\sigma_a$ linearly from $h(M)$ to $M$.
	For
	$m\in M$ and $r\in R$, let $[m,r]_a = mr -\sigma_a(r) m$.  Using this
        notation,
	\[
	\mathcal{Z}_{q}(M) = \K\spa \{ m\in h(M) \mid
	\exists a\in \Gamma\text{ such that }[m,r]_a =0 \forall r\in R \}.
	\]
 	Let $M_{q,0} = R\mathcal{Z}_q(M)R$.   For $i\geq 1$, $M_{q,i}$
        denotes the $R$-bimodule
	generated by the set
	\[
	 \K\spa\{ m\in h(M) \mid
	\exists a\in \Gamma\text{ such that }[m,r]_a \in M_{q, i-1}
        \forall r\in R \}.
	\]
Note, $M_{q,0}\subset M_{q,1}\subset \cdots $ and $M_{q-\textit{diff}} =
\cup_{i\geq 0}M_{q,i}$.
	When $M= \grHom_{\K}(R,R)$ we get the filtered algebra of quantum
        differential operators
	$D_q(R) =M_{q-\textit{diff}}$ and the $R$-bimodule of quantum
        differential operators of order $\leq i$
	is $D_{q}^i(R) = M_{q,i}$.

For each $r\in h(R)$ let $\lambda_r, \rho_r \in \grHom_{\K}(R,R)$ be
the left and right multiplication maps,
\[
\lambda_r (s) = rs \quad \rho_r (s) = sr \quad \forall s\in R.
\]
By the definition of quantum differential operators, we have:
\begin{itemize}
\item 	 The algebra $D^0_q(R)$ is generated by the set
	\[
	\{ \lambda_r,
  \rho_s, \sigma_a \mid r,s\in R, a\in \Gamma \}\quad \text{where}
	\]
	\[ \lambda_r\rho_s = \rho_s\lambda_r, \hspace{0.05in}
	\sigma_a \lambda_r = \lambda_{\sigma_a(r)} \sigma_a,
	\textit{ and }
	\sigma_a \rho_r = \rho_{\sigma_a(r)} \sigma_a.
	\]
        For any algebra $A$ and a group $G$ acting on $A$, the skew-group
        algebra $A\# G$ is the left free $A$-module with basis $\{ g \mid
        g\in G \}$ along with an associative multiplication defined by
        \[
        (a_1 g_1)( a_2 g_2 ) = a_1 g_1(a_2) g_1g_2.
        \]
       Then $D^0_q(R)$ is a homomorphic image of
       $(R\otimes_{Z(R)} R^o) \# \Gamma$ where $Z(R)$ denotes the
       centre of $R$, $R^o$ denotes the opposite algebra of $R$,
       and the homomorphism is given by
       \[
       (a\otimes b^o)\gamma \mapsto \lambda_a \rho_b\sigma_{\gamma}.
       \]
\item 	For $i\geq 1$, each $D^i_q(R)$ is the $R$-bimodule generated by the
  $\K$-span of the set
	\[
	\{ \textit{homogeneous }\varphi \mid \exists a\in \Gamma
        \textit{ such that } [\varphi ,r ]_a \in D^{i-1}_q(R) \}.
	\]
        Note that
        $[\varphi \sigma_b, r]_{a+b}=[\varphi, \sigma_b(r)]_a \sigma_b$.
	In particular,\[[\varphi \sigma_{-a}, r] = [\varphi, \sigma_a (r)]_a \sigma_{-a}.\]
	Hence, each $D^i_q(R)$ is
        the $D^0_q(R)$-bimodule generated by
	the $\K$-span of the set
	\[
	\{ \textit{homogeneous }\varphi \mid  [\varphi ,r ] \in
        D^{i-1}_q(R) \}.
	\]
\item 	For each $a \in \Gamma$, let $\varphi \in \grHom_{\K}(R,R)$ be a left
  skew $\sigma_a$-derivation. That is, $\varphi (rs) = \varphi (r)s +
  \sigma_a(r)\varphi(s)$ $\forall r,s\in R$.
	Then $[\varphi , r]_a = \lambda_{\varphi (r)}, \forall r\in R$. That
        is, $\varphi \in D^1_q (R)$. We say that $\varphi$ is a right
        skew $\sigma_b$-derivation if $\varphi (rs)=\varphi (r)\sigma_b (s)
        + r \varphi (s)$.
        That is, $[\varphi ,r]=\lambda_{\varphi (r)}\sigma_b$.
        Note that $\varphi$ is a left $\sigma_a$-derivation
        if and only if $\varphi\sigma_{-a}$ is a right
        $\sigma_{-a}$-derivation.

        Suppose $\varphi$ is a left $\sigma_a$-derivation, then
        $[\varphi , \rho_s] = \rho_{\varphi (s)}\sigma_a$.
        Similarly,
        if $\varphi$ is a right $\sigma_a$-derivation, then
        $[\varphi , \lambda_r] = \lambda_{\varphi (r)}\sigma_a$.
\end{itemize}

\begin{notation}\label{qnumber}
For $n\in \N$ and $m\in \K^*$, the $q$-\emph{number} $\qnum{n}{m}$ is defined to be $1+m+m^2+\dots+m^{n-1}$. If $m\neq 1$ then $\qnum{n}{m}=\frac{m^n-1}{m-1}$ and if $m=1$ then $\qnum{n}{m}=n$.
\end{notation}


\section{Quantum Differential Operators on Polynomial Algebras}\label{poly}
\subsection{Polynomial algebra in one variable
(\cite{IM1}) } \label{poly-one-variable-generic}
Let $\K$ be a field and let $q\in \K$ be transcendental  over $\mathbb{Q}$.
Let $R= \K [x]$ be the polynomial algebra in one variable.
$R$ is $\mathbb{Z}$-graded, with $\deg(x)=1$.  Define the
bicharacter $\beta : \mathbb{Z} \times \mathbb{Z} \to \K^*$
by $\beta (n,m) = q^{nm}$.

It is shown in \cite{IM1} that the algebra of quantum differential operators on $R$,
denoted by $D_q(R)$, is a $\K$-algebra generated by the set
\[
\{ \lambda_x = x, \del^{\beta} = \del^{\beta^1},
	\del=\del^{\beta^0}, \del^{\beta^{-1}} \}\quad \text{where}
\]
the operators are given by:
\begin{align*}
x\dot(x^n) &= x^{n+1},\\
\del^{\beta^a} (x^n) &= \qnum{n}{q^a} x^{n-1}
\textit{ for } a\in \{ -1,0, 1 \}.
\end{align*}

We now aim to determine a full set of defining relations for these generators and a basis for $D:=D_q(\K[x])$.
To facilitate the handling of monomials in the generators $x, \del, \del^\beta$ and $\del^{\beta^{-1}}$,
we rewrite $\del^\beta$ and $\del^{\beta^{-1}}$ as $\del_1$ and $\del_{-1}$ respectively. In general formulae involving
$\del_a$ for $a=-1,0,1$, $\del_0$ should be interpreted as $\del$.

The generators $x$, $\del$, $\del_1$ and $\del_{-1}$ satisfy the Weyl algebra or quantized Weyl algebra relations
\begin{equation}
\del_a x - q^a x \del_a=1,\quad a=-1,0,1, \label{gens-rels1}\end{equation}
and the relations, inserting $x$ into commutation relations,:
\begin{equation}
\del_a x \del_b =  \del_b x \del_a, \quad a\neq b\in \{-1,0,1\}.\label{gens-rels2}\end{equation}
There is also $q$-commutation between $\del_1$ and $\del_{-1}$:
\begin{equation}\del_{-1}\del_1 = q \del_1 \del_{-1}.\label{gens-rels3}\end{equation}

We shall show that \eqref{gens-rels1}, \eqref{gens-rels2} and \eqref{gens-rels3} form a full set of defining relations for $D$. In \cite{IM1}, this was claimed for \eqref{gens-rels1}, \eqref{gens-rels2}, \eqref{gens-rels3}  and an extra relation, \eqref{gens-rels4} below, which we can now derive from
the others. However the proof given in \cite{IM1} is invalid.

Let $E$ denote the $\K$-algebra generated by $x, \del, \del_1$ and $\del_{-1}$ subject to
the above relations \eqref{gens-rels1}, \eqref{gens-rels2} and \eqref{gens-rels3}. Thus there is a surjective $\K$-homomorphism $\theta:E\twoheadrightarrow D$.
There will be some abuse of notation in using the same notation for elements of $D$ and $E$. We shall show
that $\theta$ is an isomorphism. In the following lemma  $\qnum{n}{m}$ is defined as in Notation~\ref{qnumber}.
\begin{lemma}
The following relations, where $a=-1,0,1$ and $b=-1,0,1$ hold in both $D$ and $E$:
\begin{eqnarray}
x(q^a\del_a\del_b-q^b\del_b \del_a)&=&\del_a-\del_b,\label{xdd}\\
(q^a\del_a \del_b-q^b\del_b \del_a)x&=&q^a\del_a-q^b\del_b.\label{ddx}\\
\del_a x^n-q^{an}x^n\del_a &=&\qnum{n}{q^a}x^{n-1},\label{dxncomm}\\
\del_a^nx-q^{an} x\del_a^n&=&\qnum{n}{q^a}\del_a^{n-1},\label{dnxcomm}\\
(q-1)x\del_1\del_{-1}&=&\del_{1} - \del_{-1},\label{xdd1-1}\\
q\del_1-\del_{-1}&=&(q-1)\del_1 x \del_{-1},\label{gens-rels4}\\
(q-1)\del_1\del_{-1}x&=&q\del_1-q^{-1}\del_{-1}.\label{ddx1-1}
\end{eqnarray}
\end{lemma}
\begin{proof}
Both \eqref{xdd} and \eqref{ddx} are obtained from \eqref{gens-rels2}
using \eqref{gens-rels1}. The relations \eqref{dxncomm} and \eqref{dnxcomm} can be deduced inductively from \eqref{gens-rels1}. Using \eqref{gens-rels3}, the relation \eqref{xdd}, in the case $a=1$, $b=-1$, simplifies to
\eqref{xdd1-1}. By \eqref{gens-rels1} with $a=1$ and \eqref{xdd1-1} we have \eqref{gens-rels4}.
Using \eqref{gens-rels3}, the relation \eqref{ddx} in the case $a=1$, $b=-1$ simplifies to \eqref{ddx1-1}.
\end{proof}


\begin{remark}
In the initial list of relations, \eqref{gens-rels3}  can be replaced by \eqref{gens-rels4}. To see this, assume that
\eqref{gens-rels1}, \eqref{gens-rels2} and \eqref{gens-rels4} hold and
use \eqref{gens-rels4}, \eqref{gens-rels2} with $a=-1$ and $b=1$, and \eqref{gens-rels1}, with $a=-1$, to write
\begin{equation}\label{q-1}
\del_{-1}=(1+(q^{-1}-1)x\del_{-1})\del_1.
\end{equation} Then pre-multiply by $\del_1$ and apply \eqref{gens-rels2}, with $a=1$ and $b=-1$, to obtain
\begin{equation}
\del_1\del_{-1}=(\del_1+(q^{-1}-1)\del_1 x\del_{-1})\del_1=(1+(q^{-1}-1)\del_{-1}x)\del_1^2.
\end{equation}
Next apply \eqref{gens-rels1}, with $a=-1$, and \eqref{q-1} to deduce that $\del_1\del_{-1}=q^{-1}\del_{-1}\del_1$.
Thus, in our claimed full set of defining relations, \eqref{gens-rels3} may be replaced by \eqref{gens-rels4}.
\end{remark}

In $D$ and in $E$, let $\sigma_a:=1+(q^a-1)x\del_a$ for $a=1,-1$.  By \eqref{gens-rels1}, there are alternative expressions
$\sigma_a=q^{-a}(1+(q^a-1)\del_a x)$.
In $D$, $\sigma_1$ and $\sigma_{-1}$ act as inverse $\K$-automorphisms of $\K[x]$.
\begin{lemma}
\begin{enumerate}
\item $\sigma_1$ and $\sigma_{-1}$ are inverses
of each other in $E$ as well as in $D$.
\item The following relations are satisfied:
 \begin{eqnarray}\label{dbinv}
\del_{-1}&=&\sigma_{-1}\del_1,\\
\del_a \sigma_b&=&q^b\sigma_b \del_a\quad a=-1,0,1,\; b=1,-1.\label{ds}
\end{eqnarray}
\end{enumerate}
\end{lemma}
\begin{proof}
(1)
The relations $\sigma_1\sigma_{-1}=1=\sigma_{-1}\sigma_1$ can be derived from
\eqref{gens-rels1}, \eqref{gens-rels3} and \eqref{xdd1-1}.

(2) \eqref{dbinv} is simply \eqref{xdd1-1} rewritten in terms of $\sigma_{-1}$ and \eqref{ds} is easily checked using \eqref{gens-rels2}.
\end{proof}

\begin{lemma} For $a=1,-1$, the following homogeneous cubic relations hold in $D$ and in $E$.
\begin{equation}
\label{cubicq}
q^{2a}\del_a^2\del+\del^2\del_a-(q^a+1)\del\del_a\del+\del\del_a^2+q^a\del_a\del^2-2q^a\del_a\del\del_a=0.
\end{equation}
\end{lemma}
\begin{proof}
These relations, which are special cases of relations established for $D$ in \cite{J2} follow from previously listed relations on comparing the results of simplifying
the expression $(\del \del_a-q^a\del_a \del)x(\del \del_a-q^a\del_a \del)$ using \eqref{xdd} or
using \eqref{ddx}.
\end{proof}

\begin{notation}\label{Dxetal}
Here we establish notation for some algebras related to $D$.

By $G$, we denote the subalgebra of $D$ generated by $\del, \del_1$ and $\del_{-1}$.
As the operators $\del, \del_1$ and $\del_{-1}$ on $\K[x]$ each reduce degree in $x$ by one,
$G$ is $\N$-graded with $\del, \del_1$ and $\del_{-1}$ having degree $1$.
As $x$ increases degree in $x$ by one, $D$ itself is $\Z$-graded. The actions of
$\del, \del_1$ and $\del_{-1}$ on $\K[x]$ extend to actions on $\K[x,x^{-1}]$.

By $S$, we denote the subalgebra of $\End_\K(\K[x,x^{-1}])$ generated by $\del$, $\del_1$,  $\del_{-1}$ and $x^{-1}$.
Thus $S$ is also $\N$-graded and $G\hookrightarrow S$.
By $T$, we denote the subalgebra of $S$ generated by $\del_1, \del_{-1}$ and $x^{-1}$.

Let $D_x$ denote the subalgebra of $\End_\K(\K[x,x^{-1}])$ generated by $\del$, $\del_1$, $\del_{-1}, x$ and $x^{-1}$.
As $\sigma_a=1+(q^a-1)x\del_a$, $a=1,-1$, $D_x$ is generated by
$x, x^{-1}, \del, \sigma_1$ and $\sigma_{-1}$. Note that, by \eqref{ds}, $\sigma_1 \del=q^{-1}\del\sigma_1$ and that
\begin{equation}
\sigma_1 x=qx\sigma_1.\label{sx}\end{equation}
It follows that $D_x$ is a homomorphic image, with $X\mapsto \sigma_1$,
of the skew Laurent polynomial ring $A_x[X^{\pm 1};\alpha]$, where $A$ is the first Weyl algebra, generated by $\del$ and $x$,
$A_x$ is its localization at the powers of $x$ and $\alpha$ is the $\K$-automorphism of $A_x$ such that $\alpha x=qx$ and $\alpha(\del)=q^{-1}\del$.
As $A_x$ is simple with group of units $\{\lambda x^i:i\in \Z, \lambda\in \K^*\}$, $\alpha$ is not inner and, by \cite[Theorem 1.8.5]{McCR},
$A_x[X^{\pm 1};\alpha]$ is simple, whence $D_x\simeq A_x[X^{\pm 1};\alpha].$
It follows from this that $D_x$ and its subalgebras $D, G, S$ and $T$, are domains and that $D_x$ is simple. Also, by \cite[Theorems 1.12, 1.17]{GW}, $D_x$ is right and left Noetherian.

Note that $D_x$ is the localization of $D$ at the powers of $x$ which,
as a consequence of \eqref{gens-rels1}, form an Ore set in $D$. Also $D_x$ is the localization of $S$ at the powers of $x^{-1}$.

For $n\geq 0$ and an $\N$-graded algebra $H$, let $H_n$ denote the $n$-degree component of $H$.
\end{notation}

\begin{definition}
Following \cite{J2}, a \emph{special} monomial in $\del_1$ and $\del$ is a monomial of the form
$\del_1^j$ or $\del_1^j\del\del_1^k\del^\ell$ where $j, k, \ell$ are non-negative integers. Thus a special monomial
is either a standard monomial $\del_1^d\del^e$, $d, e\geq 0$, or can be obtained from such a monomial by moving one occurrence of $\del$
to the left of a power of $\del_1$. It is shown in \cite{J2} that the special monomials in $\del_1$ and $\del$
form a basis for the subalgebra of $D$ generated by $\del$ and $\del_1$. The number of special monomials
of total degree $n$ in $\del_1$ and $\del$ is $\frac{1}{2}(n^2+n+2)$, and $n+1$ of these are standard.

Let $\mathcal{G}$ denote the set $\{\del_1^j\del_{-1}^k\del^\ell: j,k,\ell\geq 0\}\cup
\{\del_1^j\del\del_1^k\del^\ell:j,\ell\geq 0, k>0\}$. Thus $\mathcal{G}$ consists of the
standard monomials $\del_1^j\del_{-1}^k\del^\ell$ in $\del_1, \del_{-1}$ and $\del$ and
the non-standard special monomials in $\del_1$ and $\del$.

Let $\mathcal{D}$ denote the set
\[\mathcal{G}\cup\{x^j\del_1^k\del^\ell:j>0,k,\ell\geq 0\}\cup\{x^j\del_{-1}^k\del^\ell:j>0,k,\ell\geq 0\}\]
and let $\mathcal{E}$ be the corresponding set of monomials in $E$.
Thus in addition to the monomials in $\mathcal{G}$, $\mathcal{D}$
contains those  standard monomials in $x, \del_1$ and $\del$ or in $x, \del_{-1}$ and $\del$ having a positive power of $x$.
\end{definition}

\begin{lemma}
The following relations hold in both $D$ and $E$:
\begin{eqnarray}
\del_{-1}\del\del_1&=&q^2\del_1\del\del_{-1},\label{ddd}\\
\del_{-1}\del^n\del_1&=&q^{n+1}\del_1\del^n\del_{-1}\text{ for }n\geq0,\label{ddnd}\\
\del \del_{-1}&=&
\del \del_1-q\del_1\del +
(q-1)\del_1\del_{-1}+
q^{-1}\del_{-1} \del,\label{5termquad}\\
\del\del_1\del_{-1}&=&q^{-1}\del_1\del_{-1}\del+q^{-1}\del\del_1^2-q\del_1^2\del-2(1-q)\del_1^2\del_{-1}.
 \label{ddqdq-1}
\end{eqnarray}
\end{lemma}
\begin{proof}
Using \eqref{dbinv}, we see that $\del_{-1}\del\del_1=\sigma_{-1}\del_1\del\del_1$. Applying \eqref{ds} and \eqref{dbinv}, we obtain \eqref{ddd}.
The same method gives the generalization \eqref{ddnd}.

Premultiplying \eqref{xdd1-1} by $\del$ and applying \eqref{gens-rels2} and \eqref{gens-rels3},
\begin{equation}\label{delxdd1-1}
(1-q^{-1})\del_{-1}x\del\del_1=\del\del_1 -\del \del_{-1}.
\end{equation}
Postmultiplying \eqref{ddx1-1} by $\del$ and applying \eqref{gens-rels2} and \eqref{gens-rels3},
\begin{equation}\label{delddx1-1}
(1-q^{-1})\del_{-1}\del x\del_1=q\del_1\del-q^{-1}\del_{-1}\del.
\end{equation}
Subtracting \eqref{delxdd1-1} from \eqref{delddx1-1} and applying \eqref{gens-rels2}, with $a=0$, and \eqref{gens-rels3} yields
\eqref{5termquad}.

In $D$, the identity \eqref{ddqdq-1} can be confirmed by calculating the action of each side on $x^n$, for $n\geq 3$, but, for $E$, we need to
    show that it follows from previously derived relations. We shall not give full details of this calculation but shall describe the steps involved.
    Observe that, of the six monomials appearing on the left hand side of \eqref{cubicq}, with $a=-1$, four contain $\del_{-1}\del$ as a submonomial.
    Using \eqref{5termquad} to substitute for $\del_{-1}\del$ in each of these, the left hand side of \eqref{cubicq}, with $a=-1$, becomes a linear combination of
    monomials from $\mathcal{G}$ together with $\del\del_{-1}\del$, $\del\del_1\del_{-1}$, $\del^2\del_1$, $\del_{-1}\del\del_{-1}$, $\del_1\del\del_{-1}$, $\del_{-1}\del_1\del_{-1}$, $\del_{-1}\del_1\del$, and $\del_{-1}\del\del_1$. The term $\del^2\del_1$ is expressed, by \eqref{cubicq}, as a linear combination of
    special monomials in $\del_1$ and $\del$ and \ref{gens-rels3} converts $\del_{-1}\del_1\del$ and $\del_{-1}\del_1\del$ to scalar multiples of standard monomials. The identity \eqref{ddd} is used to combine the terms in  $\del_{-1}\del\del_1$ and $\del_1\del\del_{-1}$ in a single term in $\del_1\del\del_{-1}$.
        As before, \eqref{5termquad} is used to substitute for $\del_{-1}\del$ in $\del\del_{-1}\del$,  $\del_{-1}\del\del_{-1}$ and $\del_1\del\del_{-1}$. The upshot of these changes is to convert
    the left hand side of \eqref{cubicq}, with $a=-1$, to a linear combination of monomials from $\mathcal{G}$ together with
    $\del\del_1\del_{-1}$. This rearranges to give \eqref{ddqdq-1}.
\end{proof}

\begin{remark} The following quadratic relation is observed in \cite{J2}:
\begin{equation}\label{6termquad}
\del \del_1-q\del_1\del +
q\del_1 \del_{-1}-q^{-1}\del_{-1}\del_1+
q^{-1}\del_{-1} \del-\del \del_{-1}=0.
\end{equation}
Using \eqref{gens-rels3}, this is equivalent to \eqref{5termquad}.
\end{remark}

\begin{lemma}\label{spanner}
\begin{enumerate}
\item $G$ is spanned by $\mathcal{G}$.
\item
$D$ is spanned by $\mathcal{D}$.
\item
$E$ is spanned by $\mathcal{E}$.
\end{enumerate}
\end{lemma}
\begin{proof}
(1) Note that \eqref{5termquad} and \eqref{ddqdq-1} can be used to express $\del\del_{-1}$ and $\del\del_1\del_{-1}$ as linear combinations of monomials from $\mathcal{G}$.
For $n\geq 0$, let $F_n$ be the subspace of $G_n$ spanned by $G_n\cap\mathcal{G}$. We show by induction on $n$ that $G_n=F_n$. This is certainly true for $n=0,1$. Let $n>1$ and suppose that $G_{n-1}=F_{n-1}$. We need to show that $gG_{n-1}\subseteq F_n$ for all
$g\in \{\del_1, \del_{-1}, \del\}$. By the inductive hypothesis it suffices to show that $gm\in F_n$ if
$m=\del_1^j\del_{-1}^k\del^\ell\in G_{n-1}$ is standard or if $m=\del_1^j\del\del_1^k\del^\ell\in G_{n-1}$ is special but not standard (so $k\neq 0$). This is clearly true for $g=\del_1$.

Consider the case where $g=\del_{-1}$. If $m=\del_1^j\del_{-1}^k\del^\ell$ is standard then $\del_{-1}m$ is standard if $j=0$ and, otherwise, by \eqref{gens-rels3},
$\del_{-1}m\in \del_1G_{n-1}\subseteq F_n$. Let
$m=\del_1^j\del\del_1^k\del^\ell$, $k\neq 0$, be special but not standard. If $j=0$ then, by \eqref{ddd},
\[\del_{-1}m=q^2\del_1\del\del_{-1}\del\del_1^{k-1}\del^\ell\in \del_1G_{n-1}\subseteq F_n.\]
If $j\neq 0$ then, by \eqref{gens-rels3},
\[\del_{-1}m=q\del_1\del_{-1}\del_1^{j-1}\del\del_1^k\del^\ell\in \del_1G_{n-1}\subseteq F_n.\]

Finally consider the case where $g=\del$. If $m$ is special then $\del m$ is a linear combination of special monomials
by \cite[Proposition 2(iii)]{J2}. It remains to show that $\del\del_1^j\del_{-1}^k\del^\ell\in F_n$
for all $j,k,\ell\geq 0$. This is true when $k=0$, because $\del\del_1^j\del^\ell$ is special. Let $k>0$ and suppose, inductively, that $\del\del_1^j\del_{-1}^d\del^\ell\in F_n$ whenever $d<k$.
If $j=0$ then, by \eqref{5termquad}
\begin{eqnarray*}
\del\del_{-1}^k\del^\ell&=&(\del \del_1-q\del_1\del +
(q-1)\del_1\del_{-1}+
q^{-1}\del_{-1} \del)\del_{-1}^{k-1}\del^\ell\\
&\in &\del_{-1}G_{n-1}+\del_1G_{n-1}+\K\del\del_1\del_{-1}^{k-1}\del^\ell\\
&\subseteq& F_n.
\end{eqnarray*}

If $j\neq 0$ then, by \eqref{gens-rels3} and \eqref{ddqdq-1},
\begin{eqnarray*}
\del\del_1^j\del_{-1}^k\del^\ell
&\in&\K\del\del_1\del_{-1}\del_1^{j-1}\del_{-1}^{k-1}\del^\ell\\
&\subseteq& \del_1G_{n-1}+\K\del\del_1^{j+1}\del_{-1}^{k-1}\del^\ell\\
&\subseteq& F_n.
\end{eqnarray*}
This completes the proof that $G_n=F_n$ and that $\mathcal{G}$ spans $G$.

(2) For $n\geq 0$, let $D^n$ be the subspace of $D$ spanned by monomials of degree $\leq n$
in the generators $x, \del, \del_1$ and $\del_{-1}.$ Thus $\{D^n:n\geq 0\}$ is a filtration of $D$ and $D^n=D^1D^{n-1}$. Let $W$ be the subspace spanned by the standard monomials $\del_1^k\del^\ell$ in $\del_1$ and $\del$, and the
standard monomials $\del_{-1}^k\del^\ell$ in $\del_{-1}$ and $\del$. We shall show that
$D=J:=\sum_{n\geq 1}^\infty x^nW+G$. It suffices to show that $hJ\subseteq J$ for all $h\in\{x, \del, \del_1, \del_{-1}\}$ because it will then follow, inductively, that $D^n\subseteq J$ for all $n$.

First consider the case where $h=x$. It is clear that $x(\sum_{n\geq 1}^\infty x^nW)\subseteq J$ so it suffices to show that $xG\subseteq xW+G$. Let $U$ be the graded subalgebra generated by $\del_1$ and $\del$ and let $W^\prime$ be the subspace of $U$ spanned by the standard monomials in $\del_1$ and $\del$. Thus $W^\prime\subseteq W$. Using \eqref{xdd}, with $a=0$ and $b=1$, it is shown in the proof of \cite[Theorem 1]{J2} that, for $n\geq 0$,
$xU_{n+1}\subseteq xW^\prime_{n+1}\oplus U_n$, whence $xU\subseteq xW^\prime\oplus U$. If $m$ is a special monomial in $\del_1$ and $\del$ then $m\in U$ so $xm\in xW^\prime\oplus U\subseteq xW+G$. Now let $m=\del_1^j\del_{-1}^k\del^\ell$ be a standard monomial. If $k=0$ then $m$ is special and $xm\in xW+G$ and, replacing $q$ by $q^{-1}$, the same is true if $j=0$.
So we can assume that $j\neq 0$ and $k\neq 0$. By \eqref{gens-rels3} and \eqref{xdd1-1}, and writing $f=\del_1^{j-1}\del_{-1}^{k-1}\del^\ell$,
\[xm=q^{1-j}x\del_1\del_{-1}f=\frac{q^{1-j}}{q-1}(\del_{-1}-\del_1)f\in G\subseteq J.\]
Thus $xG\subseteq xW+G$, from which it follows that $xJ\subseteq J+xG\subseteq J$ and that $x^nJ\subseteq J$ for all $n\geq1$.

If $h\in\{\del, \del_1, \del_{-1}\}$ it suffices to show that $hx^nW\subseteq J$ for all $n$. Using \eqref{dxncomm},
\[hx^nW\subseteq x^nhW+x^{n-1}W\subseteq x^nG+x^{n-1}G\subseteq x^nJ+x^{n-1}J\subseteq J.\]
This completes the proof that $D=\sum_{n\geq 1}^\infty x^nW+G$ and the result follows from (1).

(3) The proof of (2) uses only relations derived from \eqref{gens-rels1}, \eqref{gens-rels2} and \eqref{gens-rels3} so it also valid for $E$.
\end{proof}

The next step is to show that the elements of $\mathcal{G}$ are linearly independent. We will be working in $D$ rather than $E$ so can exploit the action on $\K[x]$ and the fact that $D$ is a domain.

\begin{lemma}\label{vn}
Let $V_n=xG_{n+1}\cap G_n$. Then $\dim V_n=\dim G_n-1$.
\end{lemma}
\begin{proof}
 By calculations in the proof of \cite[Theorem 1]{J2}, based on \eqref{xdd}, $\qnum{j}{q}\del_1^j\equiv \del^j \bmod xG\cap G$ and
$\qnum{j}{q^{-1}}\del_{-1}^j\equiv \del^j \bmod xG\cap G$. As $xG\cap G$ is a right ideal of $G$ it follows that if $\delta\in \{\del_1, \del_{-1}, \del\}$ and $m=m^\prime\delta$ is a monomial of degree $n$ in $\del_1, \del_{-1}$ and $\del$, then $m\equiv \lambda \delta^n \bmod V_n$ for some $\lambda\in \K^*$. As $\qnum{n}{q}\del_1^n\equiv \del^j \equiv \qnum{n}{-1}\del_{q^{-1}}^n\bmod V_n$ it follows that
$m\equiv \mu \del^n \bmod V_n$ for some $\mu\in \K^*$. Thus $V_n$ has codimension
$\leq 1$ in $G_n$. By the action of $\del$ on $\K[x]$, $\del^n\notin V_n$ because
$\del^n.x^n=n!$ whereas $xG_{n+1}.x^n=0$. Therefore $\dim V_n=\dim G_n-1$.
\end{proof}

Recall from \ref{Dxetal} that $S$ is the subalgebra of $\End_\K(\K[x,x^{-1}])$ generated by $\del, \del_1, \del_{-1}$ and $x^{-1}$ and $T$ is the subalgebra of $S$ generated by $\del_1, \del_{-1}$ and $x^{-1}$.
\begin{lemma}\label{TSP}
Let $P$ be the $\K$-algebra generated by $x_1, x_2$ and $x_3$ subject to the relations
\[x_2x_1=q^2x_1x_2,\quad x_3x_1=qx_1x_3\text{ and }x_3x_2=q^{-1}x_2x_3.\]
 Then $c:=x_3^2-qx_1x_2$ and $d:=x_1x_2$ are  normal in $P$ and $T\simeq P/cP$. There exist a $\K$-automorphism
$\sigma$ of $P/Pc$ and a $\sigma$-derivation $\delta$ of $P/cP$ such that $S\simeq(P/cP)[z;\sigma,\delta]$.
Also $x^{-1}$ is normal in $S$ and $S/x^{-1}S\simeq P/dP$.
\end{lemma}
\begin{proof}
It is straightforward to check that $c$ and $d$ are normal in $P$.  Let $B$ be the subalgebra generated by $x_1$ and $x_2$, a quantum plane with parameter $q^2$ and let $C$ be the localization of $B$ at the powers of $x_1$ and $x_2$, a quantum torus. It is well-known that $C$ is simple,
see for example \cite[Corollary 1.18]{GW}. The algebra $P$ and its localization $P_\mathcal{X}$ at the powers of $x_1, x_2$ and $x_3$
are, respectively, the skew polynomial ring $B[x_3;\alpha]$ and the skew Laurent polynomial ring $C[x_3^{\pm 1};\alpha]$ where $\alpha$ is the $\K$-automorphism of $B$, resp. $C$, such that $\alpha(x_1)=qx_1$ and $\alpha(x_2)=q^{-1}x_2$.

Let $Q=P/cP$ and $Q_\mathcal{X}=P_\mathcal{X}/cP_\mathcal{X}$. Every element of
$Q_\mathcal{X}$ has a unique expression of the form $ay+b$, where $a,b\in C$ and $y$ is the image of
$x_3$ in $Q_\mathcal{X}$. Let $I$ be a non-zero ideal of $Q_\mathcal{X}$ and let
$\tau(I)=\{f\in C: fy+g\in I\text{ for some }g\in C\}$. It is readily checked that $\tau(I)$ is a non-zero ideal of
$C$ whence $\tau(I)=C$ and there exists $g\in C$ such that $y+g\in I$. Also $I\cap C=0$ or $I\cap C=C$ and, in the latter case $I=Q_\mathcal{X}$.
Let $r=qx_1x_2\in C$. Then
$I\ni (y+g)y=r+gy$ and $I\ni g(y+g)=gy+g^2$. Hence $r-g^2\in I\cap C$ so either $r=g^2$ or $I=Q_\mathcal{X}$. But it is clear, by a degree argument,  that $r$ has no square root in $C$. Thus $I=Q_\mathcal{X}$ which is therefore simple.

Now let $w_1=x^{-1}\sigma_1=x^{-1}+(q-1)\del_1$ and $w_{-1}=x^{-1}\sigma_{-1}=x^{-1}+(q^{-1}-1)\del_{-1}.$ Then $w_1, w_{-1}$ and $x^{-1}$
generate $T$ while $w_1, w_{-1}, \del$ and $x^{-1}$ generate $S$. The actions of $w_1$ and $w_{-1}$ on $\K[x]$ are given by
\[w_1:x^n\mapsto q^nx^{n-1}, \quad w_{-1}:x^n\mapsto q^{-n}x^{n-1}.\]
From this it is clear that

\begin{center}
\begin{tabular}{cc}
$w_{-1}w_1=q^2w_1w_{-1}$, &$x^{-1}w_1=qw_1x^{-1}$,\\
$x^{-1}w_{-1}=q^{-1}w_{-1}x^{-1}  \textit{ and }$ &$qw_1w_{-1}=x^{-2}$.
\end{tabular}
\end{center}

Hence there is a surjective $\K$-homomorphism $\theta:Q\rightarrow T$ such that
$\theta(x_1)=w_1$, $\theta(x_2)=w_{-1}$ and $\theta(y)=x^{-1}$. Let $K=\ker\theta$ so that $T\simeq Q/K$. By the simplicity of $Q_\mathcal{X}$,
either $K=0$ or $x_1^jx_2^kx_3^\ell\in K$ for some $j,k,\ell\geq 0$. But $T$ is a domain. Clearly $x_i\notin K$, for $i=1,2,3$, so $K=0$ and
$T\simeq Q$.

From their actions on $\K[x]$, or from \eqref{gens-rels1}, \eqref{ds} and \eqref{sx}, it is easy to check the following relations between $\del$ and the generators of $T$:
\begin{eqnarray}
\del w_1&=&qw_1\del-x^{-1}w_1,\label{dwq}\\
\del w_{-1}&=&{q^{-1}}w_{-1}\del-x^{-1}w_{-1},\label{dwqm}\\
\del x^{-1}&=&-x^{-1}\del-x^{-2}\label{dxm}.
\end{eqnarray}
There is a $\K$-automorphism $\rho$ of $P$ such that $\rho(x_1)=qx_1$, $\rho(x_2)=q^{-1}x_2$ and $\rho(x_3)=x_3$. It is a routine matter to verify that there is a left $\rho$-derivation $\gamma$ such that
$\gamma(x_1)=-x_3x_1$, $\gamma(x_2)=-x_3x_2$ and $\gamma(x_3)=-x_3^2$ and that $\rho(c)=c$ and $\gamma(c)=-2x_3c\in cP$. Consequently $\rho$ and $\gamma$ induce, respectively, a $\K$-automorphism $\sigma$ of $Q$ and a $\sigma$-derivation $\delta$ of $Q$. By \eqref{dwq}, \eqref{dwqm} and \eqref{dxm} and the universal property of skew polynomial rings \cite[Theorem 2.4]{GW}, there is a $\K$-homomorphism $\phi:Q[z;\sigma,\delta]\twoheadrightarrow S$ such that
$\phi(x_1)=w_1$, $\phi(x_2)=w_{-1}$, $\phi(y)=x^{-1}$ and $\phi(z)=\del$.

The element $y$ is clearly normal in the domain $Q[z;\sigma,\delta]$.  Let $L$ be the localization of
$Q[z;\sigma,\delta]$ at the powers of $y$. As $y^2=qx_1x_2$, $x_1$ and $x_2$ are also invertible in $L$ which is generated by $z, y, u:=y^{-1}, u_1:=u^{-1}x_1, u_2:=u^{-1}x_2=u_1^{-1}$. Here $\del u-u\del=1$ so
$\del$ and $u$ generate a copy of the first Weyl algebra $A_1$. Also $u_1\del=q^{-1}\del u_1$ and $u_1u=quu_1$
so $L$ is isomorphic to the simple ring $A_1[u_1^{\pm 1};\tau]$ where $\tau(\del)=q^{-1}\del$ and $\tau(u)=qu$.
Consequently either $\ker \phi=0$ or, as $Q[z;\sigma,\delta]$ is a domain, $y\in \ker\phi$. But, by the action on $x^{-1}$ on  $\K[x]$, $\phi(y)\neq 0$ so $\phi:Q[z;\sigma,\delta]\rightarrow S$ is an isomorphism. Notice that $L\simeq D_x$.

The algebra
$Q[z;\sigma,\delta]/yQ[z;\sigma,\delta]$ is generated by (the images of) $x_1$, $x_2$ and $z$ subject to the relations
\[zx_1=qx_1z,\quad zx_2=q^{-1}x_2z\text{ and } x_1x_2=0\]
so
$Q[z;\sigma,\delta]/yQ[z;\sigma,\delta]
\simeq P/dP$.

Also $S/x^{-1}S\simeq Q[z;\sigma,\delta]/yQ[z;\sigma,\delta]\simeq P/dP$.
\end{proof}

\begin{theorem}
\begin{enumerate}
\item For $n\geq 1$, $\dim G_n=n^2+n+1$ and $\mathcal{G}$ is a basis for $G$. The Hilbert series for $G$ is
$(1+t^2)/(1-t)^{-3}$ and $\GK(G)=3$.
\item
$\mathcal{D}$
is a basis for $D$.
\end{enumerate}
\end{theorem}
\begin{proof}
(1) The number of monomials of degree $n$ in the spanning set $\mathcal{G}$ for $G$ is $n^2+n+1=\frac{1}{2}((n+2)(n+1)+n(n-1))$,  the number of standard monomials in $\del_1, \del_{-1}$ and $\del$ of degree $n$ being $\frac{1}{2}(n+2)(n+1)$ and the number of non-standard special monomials in $\del_1$ and $\del$ of degree $n$ being $\frac{1}{2}n(n-1)$. The generating function for the sequence $\{n^2+n+1\}_{n\geq 0}$ is $(1+t^2)/(1-t)^{-3}$. So it
suffices to show that $\dim G_n=n^2+n+1$.

Let $W$ denote the $\K$-subspace of $G$ spanned by the standard monomials in $\del_1$ and $\del$ and the standard monomials in $\del_{-1}$ and $\del$ and let $W_n=W\cap G_n$. Let $\theta$ be the composition
of the homomorphisms \[G\hookrightarrow S\twoheadrightarrow S/x^{-1}S\simeq P/dP,\]
where $P$ and $d$ are as in Lemma~\ref{TSP}.
Thus $\theta(\del_1)=\overline{x_1}$, $\theta(\del_{-1})=\overline{x_2}$,
$\theta(\del)=\overline{x_3}$ and $\theta(\del_1\del_{-1})=0$. The standard monomials in $\theta(\del_1)$ and $\theta(\del)$ and the standard monomials in $\theta(\del_{-1})$ and $\theta(\del)$ form a basis for
$P/dP$. It follows that the standard monomials in $\del_1$ and $\del$ and the standard monomials in $\del_{-1}$ and $\del$
are linearly independent. Hence $\dim_\K W_n=2n+1$ for all $n\geq 0$. It also follows that the restriction of $\theta$ to $W$ is injective so, as $\theta(G\cap x^{-1}S)=0$, we have $W\cap x^{-1}S=0$, whence $W\cap x^{-1}G=0$
and $xW\cap G=0$. The sum $J:=\sum_{n\geq 1}^\infty x^nW+G$ is direct. To see this, let
$J_k=\sum_{n\geq 1}^k x^nW+G$ for $k\geq 1$. We have seen that the sum $J_1+G$ is direct. Suppose that $k>1$ and that the sum $J_{k-1}$ is direct. Let $x^kw\in x^kW\cap J_{k-1}$. Then $w\in W\cap x^{-1}S=0$ so $J_k$ is direct. By induction, $J$ is a direct sum.

In the proof of Lemma~\ref{spanner}, we saw that $xG\subseteq xW+G$, whence $xG_{n+1}\subseteq xW_{n+1}\oplus G_n$.
Therefore $xG_{n+1}=xW_{n+1}\oplus V_n$, where $V_n=xG_{n+1}\cap G_n$ as in Lemma~\ref{vn}. By Lemma~\ref{vn},
$\dim_\K V_n=d_n-1$. As $x$ is invertible in $\End_\K \K[x^{\pm 1}]$,  $\dim_\K xG_{n+1}=d_{n+1}$ and
$\dim_\K xW_{n+1}=\dim_\K W_{n+1}=2n+3$. Therefore
$d_{n+1}=d_n+2n+2$. It is readily checked that $d_n=n^2+n+1$ is the solution to this difference equation, with initial condition $d_0=1$. Thus $\dim_\K G_{n}=n^2+n+1$ as required, completing the proof as indicated at the start.

(2) We have seen that $D=J=\sum_{n\geq 1}^\infty x^nW+G$ and that this sum is direct. As $x$ is regular in $\End_\K(\K[x])$ and $W$ has a basis $\mathcal{W}$ consisting of the monomials that are standard in $\del_1$ and $\del$ or in
$\del_{-1}$, each summand $x^nW$ has basis $x^n\mathcal{W}$ while $G$ has basis $\mathcal{G}$ by (1). The result follows as the sum is direct.

\end{proof}

\begin{corollary}
$D\simeq E$ is generated by the four generators $x, \del_1, \del_{-1}$ and $\del$ and presented by the  relations \eqref{gens-rels1}, \eqref{gens-rels2} and \eqref{gens-rels3}.
\end{corollary}
\begin{proof}
There is a surjective homomorphism $\theta:E\rightarrow D$. The elements of $\mathcal{D}$ are linearly independent so application of $\theta$ to any linear combination of elements of $\mathcal{E}$ shows that the same is true for
the spanning set $\mathcal{E}$ of $E$ from which it follows that $\theta$ is an isomorphism.
\end{proof}

\begin{remark}\label{moreonDxx-1}
We close this subsection with some comments on the algebra $D_q(\K [x,x^{-1}])$.
It is shown in \cite[Theorem 3.1]{IM1} that
this algebra is generated by
the set $\{ \lambda_x, \lambda_{x^{-1}}, \del, \del_{1},
\del_{-1} \}$. This is the algebra denoted $D_x$ in Notation~\ref{Dxetal} so
$D_q(\K [x,x^{-1}])\simeq A_x[X^{\pm 1};\alpha],$ a skew Laurent polynomial ring over the first Weyl algebra.
Furthermore, the relations (\ref{gens-rels1})-(\ref{gens-rels3}) along with the relation
$\lambda_x \lambda_{x^{-1}} = \lambda_{x^{-1}}\lambda_x =1$ form a full set of defining relations for these generators. Note that the proof  in \cite[Theorem 3.1]{IM1} for the case of $\K [x,x^{-1}]$ is valid
whereas the proof for $\K [x]$ is invalid.
An alternative set of generators is $\{ \lambda_x, \lambda_{x^{-1}}, \del, \sigma_{1},
\sigma_{-1}\}$.

\end{remark}
We thank S. Datt for pointing out the following:
\begin{corollary}\label{datt}
There exists an isomorphism between  $D_q(\K [x,x^{-1}])$ and its
opposite algebra, $D_q(\K [x,x^{-1}])^o$, such that the isomorphism restricts to
an isomorphism of $D_q(\K [x])$ and its opposite, $D_q(\K [x])^o$.
\end{corollary}
\begin{proof}
Use the generators to define $\Phi : D_q(\K [x,x^{-1}]) \to D_q(\K [x,x^{-1}])$
as follows:

\begin{center}
\begin{tabular}{ccc}
$\Phi (\lambda_x)=(\lambda_x)^o$,
&$\Phi (\lambda_{x^{-1}})=(\lambda_{x^{-1}})^o$,
&$\Phi (\del)=-(\del)^o$,\\
$\Phi(\del_1)=-q^{-1}(\del_{-1})^o$,
&$\Phi(\del_{-1})=-q(\del_1)^o$. &{}
\end{tabular}
\end{center}

This $\Phi$ gives the required isomorphism.  In terms of the
alternative set of generators given in Remark \ref{moreonDxx-1}
above, $\Phi (\lambda_x)=(\lambda_x)^o$,
$\Phi (\lambda_{x^{-1}})=(\lambda_{x^{-1}})^o$,
$\Phi (\del)=-(\del)^o$, $\Phi (\sigma_1) = q^{-1}(\sigma_{-1})^o$,
and $\Phi (\sigma_{-1}) = q(\sigma_{1})^o$.
\end{proof}

\subsection{The Noetherian property}\label{noeth}
This section addresses the noetherian property for some of the algebras of quantum differential
operators $D_q(R)$ previously considered. We begin with two cases where the invertibility of generators
of $R$ allows a change of generators that yields $D_q(R)$ as a $\K$-algebra that can be constructed from $\K$ by
a combination of polynomial extensions, skew polynomial extensions and skew Laurent extensions.

\subsection{The Noetherian condition and localization}
We have seen in the discussion within Notation~\ref{Dxetal} that $D_q(\K[x^{\pm 1}])$
is right and left Noetherian. This is a localization of $D_q(\K[x])$
 at the powers of a single element, $x$. In this subsection, we discuss the general situation where
$A$ is a domain and $0\neq a\in A$ is such that $\mathcal{A}:=\{a^i:i\geq 0\}$ is a right and left Ore set
and the localization $A_\mathcal{A}$ is known to be right and left Noetherian. We prove a result sufficient to show that $D_q(\K[x])$ is right and left Noetherian provided $q$ is transcendental over $\Q$.

\begin{definition}
Let $A$ be a domain, let $0\neq a\in A$ and let $C$ be a subring of $A$ such that $a\in C$. We say that $a$ is right (resp. left) $C$-seminormal in $A$ if $a$ is normal in $C$, that is $aC=Ca$, and for all $r\in A$ there exists $k\in \N$ such that $ra^k\in aC$ (resp. $a^kr\in Ca$).
\end{definition}
Note that if $a$ is right $C$-seminormal in $A$ then $\{a^i:i\in \N\}$ is a right Ore set in $A$ for, given $r\in A$ and $j\in \N$, there exist $c,c^\prime \in C$ and $k\in \N$ such that $ra^k=ac$ and $ca^{j-1}=a^{j-1}c^\prime$, whence $ra^{k+j-1}=a^jc^\prime$. Similarly, if $a$ is left $C$-seminormal then $\{a^i:i\in \N\}$ is a left Ore set in $A$. If $a$ is both right and left $C$-seminormal then $\{a^i:i\in \N\}$ is an Ore set and the right and left localizations coincide by \cite[Proposition 9.8]{GW}.

\begin{lemma}\label{seminormal}
Let $A$ be a domain, let $0\neq a\in A$ and let $C$ be a subring of $A$ such that $a\in C$ and $a$ is normal in $C$.
Let $D=\{d\in A: da^k\in aC\text{ for some }k\in \N\}$. Then $D$ is a subring of $A$ and, if $A$ is also a $\K$-algebra, $D$ is a subalgebra of $A$. Consequently if $A$ is generated, as a ring or as a $\K$-algebra,  by elements of $D$ then $a$ is right $C$-seminormal in $A$.
\end{lemma}
\begin{proof} It suffices to show that $D$ is closed under addition and multiplication.
Let $b,d\in A$ with $ba^j=ab^\prime$ and $da^k=ad^\prime$, where $b^\prime, d^\prime\in C$ and $k\geq j$.
Then $(b+d)a^k=a(d^\prime+b^\prime a^{k-j})\in aC$ so $D$ is closed under addition. Also, by the normality of $a$ in $C$, there exists $d^{\prime\prime}\in C$ such that $d^\prime a^{j-1}=a^{j-1}d^{\prime\prime}$. Hence
\[bda^{k+j-1}=bad^\prime a^{j-1}=ba^jd^{\prime\prime}=ab^\prime d^{\prime\prime}\in aC.\]
Thus $D$ is closed under multiplication and is a subring of $A$.
\end{proof}

\begin{example}
Let $T$ be a commutative domain with a derivation $\delta$, let $A$ be the Ore extension or skew polynomial ring $T[t;\delta]$ and let $a\in T\backslash{0}$. Let $C$ be the subring of $A$ generated by $T$ and $at$. Here $a$ is normal in $C$, with
$(at)a=a(at+\delta(a))$. Also $t\in D$, as $ta^2=a(at+2\delta(a))$. Thus $a$ is right $C$-seminormal in $A$. The fact that $\{a^i:i\in \N\}$ is a right Ore set in $A$ in this case is well-known and is a special case of \cite[1.4]{G}. A particular instance, instructive for quantum differential operators, is that of the Weyl algebra $\K[x,\del:\del x-x\del=1]=\K[x][\del;d/dx]$ with $a=x$ and with
$C$ generated by $x$ and $z:=x\del$.  Here $zx=x(z+2)$ and $C$ is the Ore extension
$\K[z][x;\alpha]$  where $\alpha$ is the $\K$-automorphism of $\K[z]$ such that $\alpha(z)=z-2$.
\end{example}

\begin{example}\label{DqKx}
We have seen that $D_q(K[x])$ is generated, as a $\K$-algebra, by
$x$, $\del$, $\del_1$ and $\del_{-1}.$ The subalgebra $T$ generated by $\tau:=x\del, x\del_1$ and $x\del_{-1}$, or, equivalently, by $\tau$, $\sigma:=(q-1)x\del_1+1$ and $\sigma^{-1}=(q^{-1}-1)x\del_{-1}+1$, is commutative and $x$ is normal in the subalgebra $C$ generated by $T$ and $x$, where $x\tau=(\tau-1)x$ and
$x\sigma=q^{-1}\sigma x$. Also $\del x^2=x(x\del+2)$, $\del_1x^2=x(q^2x\del_1+q+1)$ and $\del_{-1} x^2=x(q^{-2}x\del_{-1}+q^{-1}+1)$ so $D$ contains all four generators of $D_q(K[x])$ and,
by Lemma~\ref{seminormal}, $x$ is right $C$-seminormal. Similar calculations show that $x$ is left $C$-seminormal.
Thus $\{x^i:i\in \N\}$ is an Ore set in $D_q(K[x])$. As $D_q(K[x,x^{-1}])$ is generated by $D_q(K[x])$ and $x^{-1}$, the (right and left) localization of $D_q(K[x])$ at $\{x^i:i\in \N\}$ is
$D_q(K[x,x^{-1}])$.
\end{example}

In Theorem \ref{RanoethRnoeth} below we show that, under some extra conditions satisfied in Example~\ref{DqKx}, if the localization of $A$ at the powers of a seminormal element is Noetherian then so is $A$. This depends on the following ungraded version of a result of Rogalski \cite{R} on $\N$-graded rings.

\begin{theorem}[Rogalski]\label{rogalski}
Let $R\subseteq S$ be rings such that $R$ is a
left Ore domain and $S$ is left Noetherian. Then $R$
is left Noetherian if, for all nonzero elements $r$
of $R$, the left $R$-modules \[\frac{Sr\cap R}{Rr}\text{ and }
 \frac{S}{R+Sr}\] are Noetherian.
\end{theorem}
\begin{proof}
See \cite[Lemmas 5.9 and 5.10]{R}. These lemmas appear in \cite{R} in the context of  $\N$-graded rings $R$ and $S$.
Lemma 5.9 is a general result in the graded context and although Lemma 5.10 is explicitly about particular
rings $R$ and $S$ the only extra property that is used in the proof is the left Ore property for $R$. For the non-graded version, view
$R$ as $\N$-graded with $R=R_0$. In other, more informal, words ignore the references to homogeneity and shift.
\end{proof}

\begin{theorem}\label{RanoethRnoeth}
Let $a$ be a right and left $C$-seminormal element of a domain $A$ and let $\gamma$ be the automorphism of $C$ such that $ca=a\gamma(c)$ for all $c\in C$. Suppose  that the left module $A/Aa$ is simple, that $\cap_{i\geq 0}a^iA=0$ and that, for each $c\in C\backslash Ca$, there are only finitely many positive integers $j$ such that $\gamma^j(c)\in Aa\cap C$. If the localization $B$ of
$A$ at $\{a^i:i\in \N\}$ is left Noetherian then $A$ is left Noetherian.
\end{theorem}
\begin{proof}
We show that $A$ is left Noetherian using Theorem~\ref{rogalski} with $R=A$ and $S=B$. Let $r,x\in A$. As $B$ is a left Noetherian domain it is left Ore so there exist $c,d \in A$ and $m,n\geq 0$ such that
$a^{-n}cx=a^{-m}dr$. Premultiplying by $a^t$, where $t=\max\{m,n\}$, we see that $c_1x=d_1r$ for some $c_1,d_1\in A$. Thus $A$ is left Ore.

Suppose that $r\neq 0$. Consider the ascending chain
\[Ar\subseteq Aa^{-1}r\cap A\subseteq\ldots\subseteq Aa^{-j}r\cap A\subseteq Aa^{-(j+1)}r\cap A\subseteq \ldots.\eqno{(*)}\]
The union of this chain is $Br\cap A$. Also
\[\frac{Aa^{-(j+1)}r\cap A}{Ar}\hookrightarrow \frac{Aa^{-(j+1)}r}{Ar}\simeq \frac{Aa^{-(j+1)}}{A}\simeq\frac{A}{Aa^{j+1}}\]
so $\frac{Aa^{-(j+1)}r\cap A}{Ar}$ has finite length and is a Noetherian left module. To show that
$\frac{Br\cap A}{Ar}$ is Noetherian, it will suffice to show that $(*)$ terminates.

As $\cap_{i\geq 1}a^iA=0$ there exists $j\geq 0$ such that $r=a^ju\in a^jA$, where $u\in A\backslash aA$. Then $Br=Bu$ and $Aa^{-(j+t)}r\cap A=Aa^{-t}u\cap A$ for $t\geq 0$ so deleting the first $j$ terms in the chain $(*)$ gives the corresponding chain for $u$. Thus we may assume that
$r\notin aA$.
Let $k\geq 0$ be minimal such that $ra^k\in aC$ and let $c\in C$ be such that $ra^k=ac$. Note that $c\notin Ca$ otherwise $ra^{k-1}\in aC$, contradicting the minimality of $k$ unless $k=0$ and $r\in aC\subseteq aA$.

By hypothesis, there exists $N\in \N$ such that, for all $j\geq N$, $\gamma^j(c)\notin Aa$. Let $j\geq N$ and let $s=ta^{-(j+1)}r\in Aa^{-(j+1)}r\cap A$, where $t\in A$.
Then
\[sa^k=ta^{-(j+1)}ra^k=ta^{-(j+1)}ac=ta^{-j}c=t\gamma^j(c)a^{-j},\]
whence
$t\gamma^{j}(c)a^{-(j+k)}=s\in A$ and $t\gamma^{j}(c)\in Aa^{j+k}\subseteq Aa$.

Let $I$ be the left ideal $\{x\in A: x\gamma^{j}(c)\in Aa\}$. Thus $t\in I$. As $a\gamma^j(c)=\gamma^{j-1}(c)a\in Aa$, we have
$Aa\subseteq I$. However $1\notin I$ so, by the maximality of $Aa$, $I=Aa$. Thus $t=ua$, for some $u\in A$, and
$s=ua^{-j}r\in Aa^{-j}r\cap A$. It follows that $(*)$ terminates and that $\frac{Br\cap A}{Ar}$ is Noetherian.

Now consider the ascending chain
\[A+Br\subseteq Aa^{-1}+Br\subseteq \ldots \subseteq Aa^{-j}+Br\subseteq Aa^{-(j+1)}+Br\subseteq \ldots,\eqno{(**)}\]
the union of which is $B$. Also
\[\frac{Aa^{-j}}{A}\twoheadrightarrow\frac{Aa^{-j}+Br}{A+Br}\]
so $\frac{Aa^{-j}+Br}{A+Br}$ has finite length. To show that
$\frac{B}{A+Br}$ is Noetherian, it suffices to show that $(**)$ terminates.

Let $k$, $c$, $N$ and $j$ be as above, so that $\gamma^j(c)\notin Aa$ and $\gamma^j(c)a^{-(k+j)}=a^{-(j+1)}r$.
By the maximality of $Aa$, there exist $u,v\in A$ such that $1=u\gamma^j(c)+va$. Then
\begin{eqnarray*}
a^{-(k+j)}&=&u\gamma^j(c)a^{-(k+j)}+va^{-(k+j-1)}\\
&=&ua^{-(j+1)}r+va^{-(k+j-1)}\\
&\in& Br+Aa^{-(k+j-1)}.
\end{eqnarray*}
Thus $(**)$ terminates and, by Theorem~\ref{rogalski}, $B$ is left Noetherian.
\end{proof}

We now adapt methods from \cite{J2} to show that $x$ and $D=D_q(\K[x])$ satisfy the first of the extra hypotheses required for  Theorem~\ref{RanoethRnoeth} to be applicable.
\begin{proposition}\label{DDxsimple}
Let $D=D_q(\K[x])$. Then the left module $D/Dx$ is simple.
\end{proposition}
\begin{proof}
We aim to show that $D=Dx+\K[\del]$.
We claim that, for $j\geq1$,
\begin{equation}\label{claim}
\qnum{j}{q}\del^j\equiv jq^j\del_1\del^{j-1}\bmod{Dx},
\end{equation}
where $\qnum{j}{q}=\frac{q^j-1}{q-1}$.
By \eqref{ddx} with $a=0$ and $b=1$,
\begin{equation}\label{ddqx}
\del-q\del_1=(\del\del_1-q\del_1\del)x,
\end{equation}
which establishes \eqref{claim} when $j=1$. Let $j>1$ and suppose that
\eqref{claim} holds for $j-1$.
Multiplying throughout \eqref{ddqx}, on the left, by $\del^{j-1}$,
\begin{equation}\label{leftddqx}
\del^j\equiv q\del^{j-1}\del_1\bmod{Dx}.
\end{equation}
Multiplying throughout \eqref{ddqx}, on the right, by $\del^{j-1}$,
and using \eqref{dnxcomm} with $a=0$ and $n=j-1$,
\begin{eqnarray*}
\del^j-q\del_1\del^{j-1}&=&(\del\del_1-q\del_1\del)(\del^{j-1}x-(j-1)\del^{j-2})\\
&\equiv&(1-j)(\del\del_1-q\del_1\del)\del^{j-2}\bmod{Dx}\\
&\equiv&((1-j)\del\del_1\del^{j-2}+q(j-1)\del_1\del^{j-1})\bmod{Dx}\\
&\equiv&(-q^{1-j}\qnum{j-1}{q}\del^j+q(j-1)\del_1\del^{j-1})\bmod{Dx}.
\end{eqnarray*}
Simplifying, we see that $\qnum{j}\del^j=(q^{j-1}+\qnum{j-1}{q})\del^j\equiv jq^j\del_1\del^{j-1}\bmod{Dx}$,
that is \eqref{claim} holds for $j$.

To show that $D=Dx+\K[\del]$, let $m$ be a monomial of length $\ell$ in the generators $x, \del, \del_1$ and
$\del_{-1}$. It is enough to show that $m\in Dx+\K[\del]$. This is certainly true if $\ell=1$ so suppose it is true for monomials of length $\ell-1$ and let $m=gm^\prime$ where $m^\prime$ is a monomial of length $\ell-1$ and $g\in\{x,\del,\del_1,\del_{-1}\}$. By \eqref{claim} and its analogue for $\del_{-1}$, if $g\in
\{\del_1,\del_{-1}\}$ then $g\K[\del]\subseteq Dx+\K[\del]$. The same is true if $g=\del$ and, as  $x\del^{j}=\del^{j}x-j\del^{j-1}$, if $g=x$. Hence
\[m=gm^\prime\in g(Dx+\K[\del])\subseteq Dx+g\K[\del]\subseteq Dx+\K[\del].\]
The left module $\K[\del]$ over the first Weyl algebra $A_1=\K[x,\del:\del x-x\del=1]$, with $-x$ acting by differentiation with respect to $\del$ and $\del$ by left multiplication, is well known to be simple. Therefore $D/Dx$ is simple as a left $D$-module and may be identified with $\K[\del]$.
\end{proof}

\begin{remark}\label{cap}
Analogous methods, with \eqref{ddqx} replaced by
\[\del-\del_1=x(\del\del_1-q\del_1\del)x,\]
show that the right module $D/xD$ is also simple and can be identified with $\K[\del]$. For all $f(\del)\in \K[\del]$,
$f(\del)x^{d+1}=0$, where $d$ is the degree of $f(\del)$, so if $u\in \cap_{i\geq 1}x^iD$ then $u\in \ann_D\K[\del]$.
But $D$ is simple so $\ann_D\K[\del]=0$ and hence $\cap_{i\geq 1}x^iD=0$, the second hypothesis in Theorem~\ref{RanoethRnoeth}.
\end{remark}

\begin{proposition}\label{criticallyclosed}
With the notation from Example~\ref{DqKx}, let $0\neq c\in C\backslash xC$ and suppose that $q$ is transendental over $\Q$. Then $\gamma^j(c)\in Dx\cap C$ for only finitely many positive integers $j$.
\end{proposition}
\begin{proof}
Here $C=\K[\sigma^{\pm 1},\tau][x;\alpha]$ where $\alpha(\sigma)=q\sigma$ and $\alpha(\tau)=\tau+1$. Also
$\tau+1=\del x\in Dx$ and $\sigma-q^{-1}=(1-q^{-1})\del_1\in Dx$ so $Dx\cap C$ is the maximal ideal of
$C$ generated by $x$, $\tau+1$ and $\sigma-q^{-1}$. As $c\notin xC$, the constant term $c_0$ of $c$, as a polynomial in $x$ over $\K[\sigma^{\pm 1},\tau]$, is nonzero. Thus $c_0=c_0(\sigma,\tau)\in \K[\sigma^{\pm 1},\tau]$ and $\gamma^j(c)\in Dx\cap C$ if and only if
$c_0(q^{j-1},j-1)=0$. The result follows by applying the methods of \cite[Section 12]{R}, in particular \cite[12.7 and 12.8]{R}, to the homogenisation of $c_0$ in
$\K[\sigma,\tau,t]$, where $s$ is minimal such that $\sigma^sc_0\in \K[\sigma,\tau]$.
\end{proof}

\begin{corollary}\label{DqKxnoeth} If $q$ is transcendental over $\Q$ then $D_q(\K[x])$ is left and right Noetherian
\end{corollary}
\begin{proof}
This follows from Theorem~\ref{RanoethRnoeth}, Proposition~\ref{DDxsimple}, Remark~\ref{cap} and Proposition~\ref{criticallyclosed} and their right-sided analogues.
\end{proof}

\begin{remark}
 Rogalski \cite[Theorem 13.3]{R} shows that the subalgebra of $D_q(\K[x])$ generated by $\del$ and $\del_1$ is Noetherian. By \eqref{ds}, it follows that the algebra $T$ generated by $\del$, $\del_1$ and $\sigma^{\pm 1}$ is also Noetherian, being a homomorphic image of a skew Laurent polynomial ring over a Noetherian ring. By \eqref{dbinv},
$\partial_{-1}\in T$ so $D_q(\K[x])$ is generated by $T$ and $x$. Corollary~\ref{DqKxnoeth} can be deduced directly from this using \cite[Theorem~2.10]{McCR}.
\end{remark}

\subsection{Generalization to the polynomial algebra
in several variables}
Let $R= \K [x_1,\cdots ,x_n]$ which is $\mathbb{Z}^n$-graded by
	$\deg (x_i) =  e_i$ where
	$e_i$ is the standard
        basis element $(0,0,\cdots, 1,0,\cdots 0)$. Let	$q_1,\cdots,q_n\in \K$ be
	transcendental over $\mathbb{Q}$.
	Let $\beta : \mathbb{Z}^n \times \mathbb{Z}^n \to \K^*$ be such that
	\begin{equation*}
	\beta (e_i, e_j) =
		\begin{cases} q_i &\textit{ if } i =j\\
	 			1 &\textit{ if }i\neq j
		\end{cases}
	\end{equation*}
Then we know from \cite{IM1} that
	$D_q(R) $ is generated as an algebra by the set
	$\{ \lambda_{x_i}, \del_i^{\beta^{-1}}, \del_i, \del_i^{\beta} \mid 1
        \leq i \leq n \}$ and that $D_q(R)$ is a
	simple domain.
 Here, in Notation~\ref{qnumber}, $\del_i^{\beta^k} (\mathbf{x}^{\mathbf{a}}) =
   \qnum{a_i}{q_i^k}
			\mathbf{x}^{\mathbf{a}-e_i}$ where $k\in \{-1,0,1 \}$,
		$\mathbf{a}=(a_1,\cdots ,a_n)\in \mathbb{Z}^n$, and
		$\mathbf{x}^{\mathbf{a}} := x_1^{a_1}\cdots x_n^{a_n}$.

 The following relations can be seen
	\begin{align*}
	[\del_i^{\beta^k}, x_j] &= \delta_{ij}\sigma_{ke_i},\\
	[\del_i^{\beta^k},\del_j^{\beta^m}] &= 0 \textit{ for } i\neq j,\\
	[\del_i^{\beta^k}, \sigma_{\mathbf{a}}] &= 0 \textit{ when } a_i=0,\\
	\sigma_{e_i}\del_i^{\beta^k} &= q_i^{-1}\del_i^{\beta^k} \sigma_{e_i}.
	\end{align*}

Let $R_{s,n}$ be the commutative algebra
$\K [x_1^{\pm1},\ldots ,x_s^{\pm1},x_{s+1},\ldots ,x_n]$.
It is shown in \cite{IM1} that the algebra $D_q(R_{s,n})$ is generated by the set
	$\{ \lambda_{x_i}, \lambda_{x_j^{-1}},
        \del_i^{\beta^{-1}}, \del_i, \del_i^{\beta} \mid 1
        \leq i \leq n, j\leq s \}$.
and that $D_q(R_{s,n})$ is a simple domain.
Also, following Corollary \ref{datt} we see that $D_q(R_{s,n}) \cong (D_q(R_{s,n}))^o$
for every $s, 0\leq s \leq n$.

\begin{remark} It would be very interesting to know whether the hypotheses in Theorem~\ref{RanoethRnoeth}, in particular the simplicity of $A/Aa$, can be relaxed sufficiently to deal with $D_q(R)$ where $R$ is
the polynomial algebra $\K[x_1,\dots,x_n]$, $n\geq 2$ indeterminates $x_1,\dots,x_n$ and there is a Noetherian localization $D_q(\K[x_1^{\pm1},\dots,x_n^{\pm1}])$, obtained by inverting the powers of the element $x_1\dots x_n$.
This would also cover the above algebras $D_q(R_{s,n})$ and, potentially, the algebra of quantum differential operators over the coordinate algebra of the quantum plane discussed in the next section.
\end{remark}

\section{Quantum coordinate algebras and quantum tori}\label{qtorus}

\subsection{The quantum plane (\cite{IM2})}\label{section:quantum-plane}
Let $q\in \K$ be transcendental over $\mathbb{Q}$. Let
	$R = \K\langle x,y\rangle/(xy=qyx)$, the coordinate algebra of the quantum plane.
    Let $\Gamma = \mathbb{Z}^2$ with standard basis $\{ e_1, e_2 \}$. Here we summarise the results from
	\cite{IM2} on the quantum differential operators for the bicharacter $\beta : \Gamma \times \Gamma \to \K^*$ be given by
	\begin{equation*}
	\beta (e_i, e_j) =
		\begin{cases} q &\textit{ if } i =j\\
	 			1 &\textit{ if }i\neq j
		\end{cases}
	\end{equation*}
The algebra $R$ is graded by $\Gamma$, with $\deg(x)=e_1$ and  $\deg(y)=e_2$.
Note that $\lambda_x = \rho_x \sigma_{e_2}$ and $\lambda_y =
        \rho_y \sigma_{-e_1}$.
For $a=-1,0,1$, let $\del_x^{\beta^a}, \del_y^{\beta^a}\in \End_\K(R)$
	be as follows:
	\[
	\del_x^{\beta^a} (x^iy^j) = \left( \qnum{i}{q^a} \right)
        x^{i-1}y^j \quad \textit{and}\quad
	\del_y^{\beta^a} (x^iy^j) = \left( \qnum{j}{q^a} \right)
        x^{i}y^{j-1}.
	\]
 Let $D_x$ and $D_y$  be the subalgebras of $D_q(R)$ generated by
	$\{ \lambda_x, \del_x^{\beta^a} \mid a = -1,0,1 \}$ and
		 $\{ \rho_y, \del_y^{\beta^a} \mid a = -1,0,1 \}$ respectively.
As $\K$-algebras, $D_x \cong D_y \cong D_q(\K [t])$, where $D_q(\K[t])$
is as in Subsection~\ref{poly-one-variable-generic}. Further, $D_q (R)$ which is a simple domain
	and $D_x \otimes D_y \cong D_q (R)$ as filtered algebras.
	 As in the case of the commutative polynomial algebra in $\geq 2$ indeterminates, we do not know whether
$D_q (R)$ is Noetherian, though its localization at the powers of $xy$ is Noetherian
as will be shown in Theorem \ref{theorem:A1,A2,A3}.

\begin{remark}\label{different-betas}
	The bicharacter $\beta$ above, corresponds to the matrix
	$
	\begin{pmatrix}
	q &1\\
	1 &q
	\end{pmatrix}
	$.
	Denote by $\sigma_x = \sigma_{e_1}$ and $\sigma_y = \sigma_{e_2}$  for $R$ the quantum plane
	as above.  Then, $D_q^0(R)$ is the subalgebra of $\Hom (R,R)$ generated by the set
	$\{ \sigma^{\pm 1}_x, \sigma^{\pm 1}_y, \lambda_x ,\rho_y  \}$.
	
	If the bicharacter $\beta$ corresponds to the matrix
 	$\begin{pmatrix} 1& q\\ q^{-1} &1 \end{pmatrix}$ then
	$\sigma_x = \sigma_{-e_2}$ and $\sigma_y = \sigma_{e_1}$, and again we see that
	$D_q^0(R)$ is the subalgebra of $\Hom (R,R)$ generated by the set
	$\{ \sigma^{\pm 1}_x, \sigma^{\pm 1}_y, \lambda_x ,\rho_y  \}$.
         This bicharacter is  \textit{natural} as defined below in Section \ref{higher-quantum}.

	Similarly, when $\beta$ corresponds to
	$\begin{pmatrix} 1& q\\ q &1 \end{pmatrix}$, we have $\sigma_x = \sigma_{e_2}$ and
	$\sigma_y = \sigma_{e_1}$.

	Once we see that $D_q^0(R)$ is the same algebra in each case, we refer to
	 \cite[Corollary 1.1]{IM1}, which states that
	$D^i_q(R)$ is generated as a $D^0_q(R)$-module by
	the set
	$\{ \textit{homogeneous }\varphi \mid [\varphi ,r] \in D^{i-1}_q(R) \forall r\in R \}$.
	That is, in each of the three cases of $\beta$ the algebra $D_q(R)$ is the same.
\end{remark}

\subsection{Higher quantum coordinate algebras}\label{higher-quantum}

Let  $n$ be an integer, $n\geq 1$ and $q_{ij}\in \K^*$ for
$1\leq i,j \leq n$.
Let
\[
R_n= \K \langle x_1,\ldots ,x_n\rangle  /(x_ix_j - q_{ij}x_jx_i)
\] denote the
coordinate algebra of quantum affine $n$-space $\K^n$.
Throughout this section we assume that the elements $q_{ij}$ are algebraically
independent over $\Q$.

For $1\leq s\leq n$, $R_{s,n}$ will denote the localization of $R_n$ with respect to the Ore set
$T= \{ x_1^{i_1}\cdots x_s^{i_s} \mid i_1,\ldots, i_s \geq 0 \}$, that is
\[R_{s,n}=\K \langle  x_1^{\pm1}, \ldots, x_s^{\pm1},x_{s+1},
\ldots ,x_n\rangle  /
(x_ix_j -q_{ij}x_jx_i).
\]
We refer to $R_{s,n}$ as a  \emph{quantum torus}; in particular,
$R_{n,n}$ is the \emph{quantum $n$-torus}.

Note, $R_{s,n}$ is $\Z^n$-graded
with $\deg x_i=e_i$, where $e_i$ is the standard basis element $(0,\ldots ,0,1,0,\ldots ,0)$. We have the \emph{natural} bicharacter
$\beta :\Z^n \times \Z^n \to \K^*$ given by
$\beta (e_i, e_j)=q_{ij}$. Note that, in the case where each $q_{ij}=1$, the natural bicharacter is
not the bicharacter used in the previous section when $R$ was a commutative polynomial algebra or torus.

The following result tells us that if we can identify $D_q(R_n)$ then we can identify the $D_q(R_{s,n})$ for each quantum torus. So we will concentrate first on quantum coordinate algebras of $\K^n$ and then consider quantum tori.
\begin{proposition}\label{prop:localization}
For any $\varphi \in D_q(R_{s,n})$ there exists $t\in T$ such that
$t\varphi \in D_q(R_n)$.
\end{proposition}
\begin{proof}
This is seen by induction on the order of $\varphi$. For
$\varphi \in D^0_q(R_{s,n})$, $\varphi$
is given by multiplications and grading maps and hence the result is
true.

Let $i\geq 1$ and suppose that
$\varphi \in D^i_q(R_{s,n})$ is such that $[\varphi ,r^{\prime}] \in
D_q^{i-1}(R_{s,n})$
for all $r^{\prime} \in
R_{s,n}$. Replacing $\varphi$ by $\varphi - \rho_{\varphi (1)}$, we
can assume that
$\varphi (1)=0$.
For each $l\leq n$,
$[\varphi ,x_l]\in
D_q^{i-1}(R_{s,n})$ so there exists $t_l\in T$ such that $t[\varphi ,x_l]\in
D_q^{i-1}(R)$
Hence there exists $t=x_1^{i_1}\ldots x_s^{i_s}\in T$ such that $t[\varphi ,x_l]\in
D_q^{i-1}(R)$ for all $l\leq n$.
Let $c = i_1e_1+\ldots +i_se_s$. Then
\[
[(t\varphi), x_l]_c =t[\varphi ,x_l]   \in D^{i-1}_q (R).
\]
Further, $(t\varphi) \in
\grHom_{\K}(R,R)$
and
hence $(t\varphi) \in D_q^i(R)$.
\end{proof}

\begin{remark}\label{remark:localization}
Proposition \ref{prop:localization} is true even when $\beta \equiv 1$
for the quantum torus, or for the polynomial algebra and any general $\beta$
by setting $c=0$ in the proof above.
\end{remark}

\subsection{The quantized coordinate algebra
on more than two
variables.}
Let $R_n=\K\langle x_1,\cdots ,x_n\rangle/(x_ix_j = q_{ij}x_jx_i)$ denote the
coordinate algebra of quantum $n$-space. The algebra $R_n$ is $\Z^n$-graded,
and we shall use the natural bicharacter $\beta$.

\subsubsection{\textbf{The algebra $D_q^0(R)$}}
As explained in Section~\ref{Prelim}, the algebra $D_q^0(R)$
is a homomorphic image of $(R\otimes_{Z(R)}R^o)\# \Z^n$, where
$
(a\otimes b^o)\gamma \mapsto \lambda_a\rho_b\sigma_{\gamma}.
$
Note that $Z(R)=\K$. Further, $\lambda_{x_i}=\rho_{x_i}\sigma_{e_i}$ for
$1\leq i\leq n$. Hence, $D_q^0(R)$ is generated as an algebra by the
set $\{ \lambda_{x_i}, \sigma_{e_i},\sigma_{-e_i} \mid 1\leq i \leq n \}$.
That is, $D_q^0(R) $ is a homomorphic image of $R\# \Z^n$.

Suppose
$\sum_i \lambda_{r_i} \sigma_{a_i}  =0$, then by degree considerations
and using the fact that $R$ is a domain, we may assume that
$\psi=\sum_i \alpha_i \sigma_{a_i} =0$. Since the parameters $q_{ij}$ are
algebraically independent, applying $\psi$ on the monomials $ x^iy^j $,
we get $\alpha_i =0\forall i$. That is, $D_q^0 (R) \cong R \# \Z^n$.
In particular, $D_q^0(R)$ is
a domain, free over $R$ with basis $\{ \sigma_a \mid a\in \Z^n \}$.

\subsubsection{\textbf{The $R$-bimodule $D_q^1(R)$}}

\newcommand{\Rn}{\ensuremath{\K\langle x_1,
    \ldots,x_n \rangle /(x_i x_j = q_{ij} x_j x_i)}}
\newcommand{\Rxy}{\ensuremath{\K \langle x, y \rangle /(xy = q yx
    )}}

\begin{proposition}
  Let $\Phi$ be an endomorphism of
  $R_n$ and, for $1\leq i\leq n$, let $\varphi_i=[ \Phi, \lambda_{x_i}]$.
  Then
$
   \varphi_i \lambda_{x_j} - q_{ij} \lambda_{x_j} \varphi_i =
  q_{ij} \varphi_j \lambda_{x_i} - \lambda_{x_i} \varphi_j
$.
\end{proposition}
\begin{proof} This follows from $[ \Phi , \lambda_{x_ix_j} ] = q_{ij}[ \Phi, \lambda_{x_jx_i} ]$.
\end{proof}
The next result reveals a significant difference between the case $n=2$ and the case $n\geq 3$.
\begin{corollary}\label{Cor:a=e1+me_2,Cor:a=e1}
  Let $\Phi \in \End(R)$ and , for $1\leq i\leq n$, let $\varphi_i=[ \Phi, \lambda_{x_i}]$.
 \begin{enumerate}
   \item
  Let $n=2$ and suppose there exists $a\in \Z^n$ such that $\phi_1 = \sigma_a$, and $\phi_2 = 0$.
  Then there exists $m\in \Z$ such that $a=e_1+m e_2$. Similarly, if
  $\phi_1=0$, and $\phi_2= \sigma_a$ for some $a\in \Z^n$ then there exists $m\in \Z$ such that
  $a=me_1+e_2$.
  \item
  Let $n \geq 3$ and suppose there exist $a\in \Z^n$ and $1\leq i\leq n$ such that $\phi_i = \sigma_a$,
  and $\phi_j = 0$ for $j\neq i$.  Then $a = e_i$.
  \item Let $n\geq 3$ and let $1\leq i\leq n$. There is a right $\sigma_{e_i}$-derivation $\delta_i$ of $R_n$
  such that $[ \delta_i, \lambda_{x_i}]=\sigma_{e_i}$ and $[\delta_i, \lambda_{x_j}]=0$ if $j\neq i$.
    \end{enumerate}
\end{corollary}
\begin{proof}
  \begin{enumerate}
  \item
   We have $\sigma_a \lambda_{x_2} - q_{12} \lambda_{x_2} \sigma_a = 0$.
  Thus $\sigma_a(x_2) = q_{12} x_2$, $\beta(a, e_2) = q_{12}$,
   $\beta(a - e_1, e_2) = 1$ and $a - e_1 = m e_2$ for some $m$. The second claim is proved similarly.
\item
  Choose $j_1, j_2$ such that $i$, $j_1$ and $j_2$ are distinct. Since $\phi_{j_1} = \phi_{j_2} = 0$,
  we can apply the
  reasoning in the proof of (1) to conclude that
  $\beta(a - e_{i}, e_{j_1}) = 1$ and $\beta(a - e_i, e_{j_2}) = 1$.
  Thus there are integers $m_1$ and $m_2$ such that
  $a = e_i + m_1 e_{j_1}$ and $a = e_i + m_2 e_{j_2}$.
  It follows that $m_{j_1} = m_{j_2} = 0$.
\item Set $\delta_i(x_i)=1$ and $\delta_i(x_j)=0$ if $j\neq i$. Then $\delta_i$ extends to a right $\sigma_{e_i}$-derivation $\delta_i$ of $R_n$ such that, for $j\neq i$,
    $\delta(x_ix_j)=q_{ij}x_j=q_{ij}\delta(x_jx_i)$, so that the defining relations of $R_n$ are respected
    and $\delta_i$ is indeed a right $\sigma_{e_i}$-derivation of $R_n$.
    It is a routine matter to verify that $[ \delta_i, \lambda_{x_i}]=\sigma_{e_i}$ and $[\delta_i, \lambda_{x_j}]=0$ if $j\neq i$.
\end{enumerate}
\end{proof}
\begin{definition}
 Let $n\geq 3$ and let $1\leq i\leq n$. We shall retain the notation of  Corollary~\ref{Cor:a=e1+me_2,Cor:a=e1}(3).
 Thus $\delta_i$ is the
   right $\sigma_{e_i}$-derivation $\delta_i$ of $R_n$ such that $\delta_i(x_i)=1$ and $\delta_i(x_j)=0$ if $j\neq i$.
 In general, for $1\leq i \leq n$, .
\[
\delta_i (x_1^{m_1}x_2^{m_2}\cdots x_n^{m_n} )
= m_i \left(\prod_{j>i} q_{ij}^{m_j}
\right) x_1^{m_1}\cdots x_{i-1}^{m_{i-1}}x_i^{m_i-1}
	x_{i+1}^{m_{i+1}}\cdots x_n^{m_n}.
\]
Let $\del_i:=\sigma_{-e_i}\delta_i$. Then $\del_i$ is a left $\sigma_{-e_i}$-derivation of $R_n$ such that
$\del_i(x_i)=1$, $\del_i(x_j)=0$ if $j\neq i$ and
\[
    [\del_i, \rho_{x_j}]  = \begin{cases} \sigma_{-e_i}& \mbox{ if }i = j \\ 0 & \mbox{ if } i \neq j.
    \end{cases}
\]
\end{definition}

\begin{theorem}\label{thm:1st order}
The left $D^0_q(R)$-module $D^1_q(R)$ is generated by the set
$\{ \delta_i \mid 1\leq i \leq n \}\cup \{ 1 \}$.
\end{theorem}
\begin{proof}
Let $\varphi \in D^1_q(R)\backslash D_q^0(R)$ be such that
$[\varphi ,x_i] \in D^0_q(R)$ for $1\leq i \leq n$.
By considering $\varphi-\rho_{\varphi (1)}$, we assume that
$\varphi (1)=0$.
Let $\deg\varphi =
\sum r_i e_i$ for $r_i \in \Z$. Note,
$\deg[\varphi ,x_j]= (r_j+1)e_j + \sum_{i\neq j} r_ie_i$.
Since $[\varphi ,x_j]\in D_q^0(R)$, we have $r_j\geq -1$ for every $j$.

\noindent
\textbf{Case 1:} Suppose that $r_j=-1$ for some $j$. Then for $i\neq j$, some
monomial $t_i \in R$, and $\alpha_{a_i}\in \K$, we have
$[\varphi ,x_i] = \rho_{t_i}
\left( \sum_{a_i\in \Z^n} \alpha_{a_i}\sigma_{a_i} \right)$ implies that
$t_i=0$ since $t_i$ cannot have degree $(-1)$ in the variable $x_j$.

Further, $r_i\geq 0$ for $i\neq j$. That is,
\[
[\varphi ,x_j] =\rho_t \left( \sum_{a\in \Z^n} \alpha_a \sigma_a \right),
\quad [\varphi ,x_i]=0 \quad \textit{for } i\neq j,\quad
\textit{where } t= \prod_{i\neq j}x_i^{r_i}.
\]
To avoid too much notation, without loss of generality consider the case
\[
[\varphi ,x_1] =\rho_t \left( \sum_{a\in \Z^n} \alpha_a \sigma_a \right),
\quad [\varphi ,x_i]=0 \quad \textit{for } i\neq 1.
\]
Note, $[\varphi ,x_1x_i]=q_{1i}[\varphi ,x_ix_1]$ for $i\neq 1$
implies that \[\rho_t \left( \sum_{a\in \Z^n}\alpha_a \sigma_a \right) x_i
=q_{1i}x_i \rho_t \left( \sum_{a\in \Z^n} \alpha_a \sigma_a \right).\]
Hence, \[
x_i\rho_t \left( \sum_{a\in \Z^n}\alpha_a \beta(a,e_i)\sigma_a \right)
=x_i\rho_t \left( \sum_{a\in \Z^n} q_{1i}\alpha_a \sigma_a \right).
\]
Since
$D^0_a(R)$ is a domain and the set $\{ \sigma_a \}_{a\in \Z^n}$ is
linearly independent over $\K$, we have
$\alpha_a \beta(a,e_i) = \alpha_a q_{1i}$.

Thus, whenever $\alpha_a \neq 0$, $\beta (a,e_i) = q_{1i}$. Hence
$a=e_1+m_{a_i}e_i$ for some $m_{a_i} \in \Z$, for every $i, i\neq 1$.
Since $n\geq 3$, this implies that $a=e_1$.
That is, $[\varphi ,x_1]=\rho_t\alpha_{e_1} \sigma_{e_1}$ and
$[\varphi ,x_i]=0$ for $i\neq 1$. Since $\varphi (1)=0$
this implies that $\varphi = \alpha_{e_1} \rho_t\delta_1$.

\noindent
\textbf{Case 2:} Assume that $r_j\geq 0$ for every $j, 1\leq j \leq n$.
Let $[\varphi ,x_i]=\rho_{t_i}
\left( \sum_{a\in \Z^n}\alpha_{a}^i\sigma_{a} \right)$ for $i\leq n$.
For each $i$, let $\left( \dfrac{t_i}{x_i} \right)
:= x_1^{r_1}\cdots x_n^{r_n}$.
By relacing $\varphi$ by
$\varphi - \sum_i \alpha_{e_i}^i \rho_{t_i}\sigma_{e_i}$, we
assume that
\[
[\varphi ,x_i]=\rho_{t_i}
\left( \sum_{a\in \Z^n\setminus \{ e_i \} }
  \alpha_{a}^i\sigma_{a} \right)
 \quad \textit{for }i\leq n.
\]
Without loss of generality assume that $[\varphi ,x_1]\neq 0$.
Suppose $\alpha^1_a \neq 0$ for some $a\notin \Z e_1$. Then,
there exists $c\in \K^*$ such that
$\psi_a = c\rho_{\left(\frac{t_1}{x_1}\right)} \sigma_{a-e_1}$
has the property $[\psi_a, x_1] = \alpha^1_a \rho_{t_1} \sigma_a$.
Therefore, replacing $\varphi$ by $\varphi -\sum_{a\notin \Z e_1, \alpha^1_a \neq 0} \psi_a$
we have two subcases:

\textbf{Subcase 1:}
$[\varphi ,x_1] = \rho_{t_1} \sum_{n\geq n_0, n\neq 1} \alpha_{ne_1}^1
\sigma_{ne_1} $, and $\alpha_{n_0e_1}\neq 0$.
Here, if $[\varphi ,x_i]=0$ for every $i\neq 1$, then by considering the
equation $[\varphi ,x_1x_i]=q_{1i}[\varphi ,x_ix_1]$ we see that
$\alpha_{ne_1}^1 = 0$ whenever $n\neq 1$. This contradicts our assumption.

Hence we can assume that
$[\varphi ,x_2] = \rho_{t_2} \sum_{a\in \Z^n\setminus \{e_2\}}
\alpha^2_{a} \sigma_a \neq 0$.
Now consider the equation
$[\varphi ,x_1x_2]=q_{12}[\varphi ,x_2x_1]$. Note that
$\lambda_{x_i} = \rho_{x_i}\sigma_{e_i}$. Then
there exist $c_{12}, c_{21} \in \K^*$ such that
\[
[\varphi ,x_1x_2] =
c_{21}\rho_t \left( q_{12}^{n_0}\alpha^1_{n_0e_1} \sigma_{n_0e_1+e_2}+\ldots
  \right) + c_{12} \rho_t \left( \alpha^2_a\sigma_{a+e_1}+\ldots \right)
\quad \textit{while}
\]
\[
q_{12}[\varphi,x_1x_2] = c_{12}\rho_t
\left(q_{12}q_{a1}\alpha^2_a\sigma_{a+e_1}+\ldots \right)+
c_{21}\rho_t \left(q_{12}\alpha_{n_0e_1}^1 \sigma_{n_0e_1+e_2}
+ \ldots\right),
\]
where $q_{a1}=\beta(a,e_1)$ and $t=x_1^{r_1+1}x_2^{r_2+1}x_3^{r_3}\cdots
x_n^{r_n}$.

Note, if $n_0e_1+e_2=
a+e_1$ for some $a$ such that $\alpha^2_a \neq 0$, then
$a= (n_0-1)e_1+e_2$ and $q_{a1}=q_{21}$. Thus,
we get $c_{21}q_{12}^{n_0} \alpha^1_{n_0e_1} +c_{12}\alpha^2_a=
c_{12}q_{12}q_{a1}\alpha^2_a +c_{21}q_{12}\alpha^1_{n_0e_1}$. Note,
$q_{12}q_{a1}=1$. Thus, we see that $n_0=1$ which contradicts our assumption.

Hence $n_0e_1+e_2 \neq a+e_1$ with $\alpha^2_a \neq 0$. Then, we have
$c_{21}q_{12}^{n_0} \alpha^1_{n_0e_1} =
c_{21}q_{12}\alpha^1_{n_0e_1}$, which also implies that $n_0=1$.
We therefore move to the next subcase.

\textbf{Subcase 2:} $[\varphi ,x_1]=0$. For any $i$ such that
$[\varphi ,x_i]\neq 0$, we consider the equation $[\varphi ,x_1x_i]
=q_{1i}[\varphi ,x_ix_1]$. This gives rise to the equation
$x_1\rho_{t_i}\left( \sum_{a\in \Z^n\setminus \{e_i \}} \alpha^i_a \sigma_a
\right) =
x_1\rho_{t_i}\left( \sum_{a\in \Z^n\setminus \{e_i\} }q_{1i}q_{a1}
\alpha^i_a \sigma_a
\right)$, thus implying that $q_{a1}=q_{i1}$; in other words, $a=e_i+ne_1$
for some $n\in \Z$ dependent on $a$.  In particular, if
$[\varphi ,x_2]\neq 0$
then
$[\varphi ,x_2]= \rho_{t_2}\sum_{
  n\in \Z\setminus \{ 0 \}} \alpha^2_n
\sigma_{ne_1+e_2} $. For any $n$ such that
$\alpha^2_n\neq 0$,  there exists $c_n \in \K^*$ such that
$\psi_n := c_n \rho_{\frac{t_2}{x_2}} \sigma_{ne_1}$ has the property
that $[\psi_n ,x_2] = \alpha^2_n \rho_{t_2}\sigma_{ne_1+e_2} $;
note that $[\psi_n,x_1]=0$. Thus, replacing $\varphi$ by
$\varphi - \sum_{\alpha^2_n\neq 0} \psi_n$ we get
that $[\varphi,x_1]=[\varphi ,x_2]=0$.
This in turn implies
for
$i\notin \{ 1,2 \}$ (by considering
$[\varphi ,x_2x_i] = q_{2i}[\varphi ,x_ix_2]$), that
whenever $\alpha^i_a\neq 0$ we get
$a=e_i+me_2$
and (by considering $[\varphi, x_1x_i]  =q_{1i}[\varphi ,x_ix_1]$)
 $a=e_i +ne_1$
for some $m,n\in \Z$ dependent on $a$. Thus, $a=e_i$, which
contradicts our assumptions.  Hence, $[\varphi ,x_1]=[\varphi ,x_2]=0$
implies that $[\varphi ,x_i]=0$ for every $i\notin \{ 1,2 \}$.
Thus, $\varphi = \rho_t$.
\end{proof}
\subsubsection{\textbf{The algebra $D_q(R)$}}
We first present some relations among the operators $\lambda_{x_i},\delta_i,
\sigma_{e_i}, \sigma_{e_i}^{-1}$ for  $1\leq i \leq n $.

\begin{center}
\begin{tabular}{ccc}
$\lambda_{x_i}\lambda_{x_j} = q_{ij}\lambda_{x_j}\lambda_{x_i}$,
&$\delta_i\lambda_{x_j} - \lambda_{x_j}\delta_i = \delta_{ij}\sigma_{e_i}$,
&$\sigma_a\lambda_{x_i} = \beta (a,e_i) \lambda_{x_i}\sigma_a$;\\
$\delta_i \delta_j = q_{ji}\delta_j \delta_i$,
&$\sigma_a \delta_i = \beta(a,-e_i)\delta_i \sigma_a$,
&$\sigma_a\sigma_b = \sigma_b\sigma_a$.
\end{tabular}
\end{center}
for $a\in \Z^n$, where $\delta_{ij}$ denotes the Kronecker-delta.
Note that \[(\rho_{x_i}\delta_i)(x_1^{m_1}x_2^{m_2}\cdots x_n^{m_n})
=m_i(x_1^{m_1}x_2^{m_2}\cdots x_i^{m_i }\cdots x_n^{m_n}).\]

Let $\Delta$ be the subalgebra of $D_q(R)$ generated by the set
$\{ \rho_{x_i},\delta_i \mid 1\leq i \leq n \}$.
The relations among these generators are, for $1\leq i,j \leq n$,:
\begin{align}
\rho_{x_i}\rho_{x_j} &= q_{ji} \rho_{x_j}\rho_{x_i},\\
\delta_i \delta_j &= q_{ji} \delta_j \delta_i,\\
\delta_i \rho_{x_j} &=q_{ij}\rho_{x_j}\delta_i\quad \mbox{ if $i\neq j$},\\
\delta_i \rho_{x_i} -\rho_{x_i}\delta_i&=1.
\end{align}
\begin{notation}
For a multi-index $I=(i_1,\ldots,i_n)$, where each $i_j\geq0$, let $x_I$, $\delta_I$ and  $\rho_I$ denote
$x_1^{i_1}\cdots x_n^{i_n}$,
$\delta_1^{i_1}\cdots \delta_n^{i_n}$
and
 $\rho_{x_1^{i_1}}\cdots \rho_{x_n^{i_n}}$ respectively.
Further, let $\sum I = i_1+\ldots +i_n$
and $I! = \prod_{j=1}^n i_j!$.
Note that
$\delta_I (x_I) = I!$ and $\delta_I (x_J) = 0$ if any $j_r<i_r$.

Define the \emph{degree} of $\rho_I \delta_J$ to be $(I - J) \in \Z^n$.
Define the \emph{standard degree} of $\rho_I\delta_J$ to be
$\sum I + \sum J$.
When appropriate, we use the lexicographical ordering on $\Z^n$.
\end{notation}

\begin{theorem}\label{Deltasimple}
The algebra $\Delta$ is a simple, left and right Noetherian, domain of GK-dimension $2n$ with
basis $\{\rho_I \delta_J\}$.
\end{theorem}
\begin{proof}
Let $\Lambda=(\lambda_{ij})$ be a $n\times n$ matrix over $\K$, with non-zero entries, such that
$\lambda_{ji}=\lambda_{ij}^{-1}$ for $1\leq i,j\leq n$ and $\lambda_{i,i}=1$ for $1\leq i\leq n$.
Let $A_n^{\Lambda}$ denote the $\K$-algebra  generated by $u_i$ and $v_i$, $1 \leq i\leq n$, subject
to the relations
\begin{eqnarray*}
v_{j}v_{i}&=&\lambda_{ji}v_{i}v_{j},\quad 1\leq i,j\leq n;\\
u_{j}u_{i}&=&\lambda_{ji}u_{i}u_{j},\quad 1\leq i,j\leq n;\\
u_{j}v_{i}&=&\lambda_{ij}v_{i}u_{j},\quad 1\leq i\neq j\leq n;\\
u_{i}v_{i}-v_{i}u_{i}&=&1,\quad 1\leq i\leq n.
\end{eqnarray*}
This is an example of the higher quantized Weyl algebras studied in \cite{JW} and it is an iterated skew polynomial ring over $\K$, constructed by adjoining $u_1,v_1,\dots, u_n,v_n$ in turn, using automorphisms and skew derivations. It is therefore right and left Noetherian and has basis $\{u_1^{a_1}\dots u_n^{a_n}v_1^{b_1}\dots v_n^{b_n}|a_i,b_i\geq 0\}$. It has an obvious filtration, with $\deg u_i=\deg v_i=1$, for which the associated graded ring $\gr (A_n^{\Lambda})$ is generated by $\overline{u}_i$ and $\overline{v}_i$, $1\leq i\leq n$ subject to the relations
\begin{eqnarray*}
\overline{v}_{j}\overline{v}_{i}&=&\lambda_{ji}\overline{v}_{i}\overline{v}_{j},\quad 1\leq i,j\leq n;\\
\overline{u}_{j}\overline{u}_{i}&=&\lambda_{ji}\overline{u}_{i}\overline{u}_{j},\quad 1\leq i,j\leq n;\\
\overline{u}_{j}\overline{v}_{i}&=&\lambda_{ij}\overline{v}_{i}\overline{u}_{j},\quad 1\leq i,j\leq n.
\end{eqnarray*}
This is the coordinate algebra of a quantum space and has the same growth as the commutative polynomial algebra in $2n$ indeterminates.
The filtration is standard and finite-dimensional in the sense of \cite{McCR} and it follows from \cite[Proposition 8.1.14]{McCR} that \[\GK(A_n^{\Lambda})=\GK(\gr(A_n^{\Lambda}))=2n.\]

It is shown in \cite[Theorem 6.5]{JW} that, provided $\ch \K=0$, $A_n^{\Lambda}$ is simple for all $\Lambda$. Alternatively, it is not difficult to adapt the well-known proof of simplicity of the usual Weyl algebra $A_n$, where each $\lambda_{ij}=1$ so as to apply to $A_{n}^{\Lambda}$.

If $\lambda_{ij}=q_{ji}$,  there is a surjective $\K$-algebra homomorphism $\theta:A_n^{\Lambda}\to \Delta$ given by
$\theta(u_i)=\delta_i$ and $\theta(v_i)=\rho_{x_i}$. As $A_n^{\Lambda}$ is simple, $\theta$ is an isomorphism and the result follows.

\end{proof}

\begin{remark}
\begin{enumerate}
\item The group $\Z^n$ acts on $\Delta$ by
$\gamma \cdot \rho_{x_i} = \beta(\gamma, e_i) \rho_{x_i}$
and
$\gamma \cdot \delta_i = \beta(\gamma, e_i)^{-1} \delta_i$ so we can form the skew group algebra
$\Delta \#\Z^n$. The algebra $\Delta \# \Z^n$ can be viewed as an iterated skew Laurent polynomial extension of $\Delta$.

\item There is a homomorphism of algebras
\[
f: \Delta \# \Z^n \to D_q(R) \quad \textit{with}
\quad
f(\varphi \gamma) \mapsto \varphi \sigma_{\gamma} \quad
\textit{for } \varphi \in \Delta, \gamma \in \Z^n.
\]
\end{enumerate}
\end{remark}

\begin{theorem}\label{DZnsimple}
The algebra $\Delta \# \Z^n$ is a
left and right  Noetherian simple domain with basis $\{ \rho_I \delta_J \sigma_{\gamma} \}$ and GK-dimension $3n$.
\end{theorem}
\begin{proof}
By   \cite[Theorem 1.5.12 and Theorem 1.4.5]{McCR},  $\Delta \# \Z^n$ is
a right and left Noetherian domain. By the condition that the $q_{ij}$'s generate a free abelian group, $\Z^n$ acts as a group of outer automorphisms so, by \cite[Proposition 7.8.12]{McCR}, $\Delta \# \Z^n$ is simple.  Simplicity can also be established by successively applying \cite[Theorem 1.8.5]{McCR} or \cite[Theorem 1.17]{GW}.
The stated set is clearly a basis and it is easy to see that $\GK(\Delta \# \Z^n)=3n$.
\end{proof}

\begin{remark}
\begin{enumerate}
\item The algebra $\Delta \# \Z^n$ is generated by the set
	\[ \{ \rho_{x_i}, \delta_i ,\sigma_{e_i}^{\pm1}
	\mid 1\leq i \leq n \}.
	\]
	Note that $\lambda_{x_i} = \rho_{x_i} \sigma_{e_i}$ so
	$\Delta \# \Z^n$ is also generated by the set
	$\{ \lambda_{x_i}, \delta_i ,\sigma_{e_i}^{\pm1}
\mid 1\leq i \leq n \}$ or by $\{ \lambda_{x_i}, \del_i ,\sigma_{e_i}^{\pm1}
\mid 1\leq i \leq n \}$, where $\del_i$ is the left $\sigma_{-e_i}$-derivation $\sigma_{-e_i}\delta_i$.

The subalgebra $S_n$
of $D_q(R)$ generated by the set
$\{\lambda_{x_i}, \del_i \mid 1\leq i \leq n\}$
is also isomorphic to a quantized Weyl algebra of the form $A_n^\Lambda$ and hence is simple.
\item
The algebra $D_q(R)$ has a filtration $D_q^0(R)\subset D_q^1(R)\subset \ldots$.
The algebra $\Delta$ is also filtered, with filtration $\Delta_0 \subset \Delta_1 \subset \ldots$
where $\Delta_m = \K$-Span$\{ \rho_I \delta_J \mid \sum J \leq m \}$. This induces a filtration
on $\Delta \# \Z^n$ given by $\left( \Delta \# \Z^n \right) _m= \Delta_m \# \Z^n$.
\item
The following formula and lemmas will be needed in the next theorem: for $1\leq i \leq n$,
\begin{align}\label{commutator}
[\delta_1^{j_1}\delta_2^{j_2}\cdots \delta_n^{j_n} \sigma_a ,x_i]
=& j_i   \left( \prod_{i<s}q_{s,i}^{j_s} \right) \beta (a,e_i)
\delta_1^{j_1}\cdots \delta_i^{j_i-1}\cdots \delta_n^{j_n} \sigma_{a+e_i}\notag \\
&+( \beta(a,e_i) -1) \left( \prod_{1\leq s \leq n} q_{s,i}^{j_s}\right)
\rho_{x_i}\delta_1^{j_1}\delta_2^{j_2}\cdots \delta_n^{j_n} \sigma_{a+e_i}.
\end{align}
\end{enumerate}
\end{remark}
\begin{lemma} \label{Lm1:Dq-qtorus}
Let $\varphi \in D_q(R)$ be such that $[\varphi ,x_1] \in f(\Delta \# \Z^n)$ and
$[\varphi ,x_i]= 0 \quad \forall i>1$.  Then $[\varphi ,x_1]= T \sigma_{e_1}$
where $T$ is in the subalgebra of $f(\Delta \# \Z^n)$ generated by the set
$\{ \rho_{x_i} \}_{1\leq i \leq n} \cup \{ \delta_1 \}$.
\end{lemma}
\begin{proof}
Let $[\varphi ,x_1] = \sum_{I,J,a}\alpha_{I,J,a}\rho_I\delta_J\sigma_a$.
Without loss of generality assume that
$\varphi$ is homogeneous.
For each tuple $(I,J,a)$ there exists $\psi_{I,J,a} \in \grHom(R, R)$ such that
 $[\psi_{I,J,a}, x_1] = \alpha_{I,J,a}\delta_J\sigma_a$
(that is, $[\rho_I\psi_{I,J,a}, x_1] = \alpha_{I,J,a}\rho_I\delta_J\sigma_a$) and
$[\psi_{I,J,a},x_i]=0 \quad \forall i>1$.

We may therefore assume that $[\varphi ,x_1]= \alpha \delta_J \sigma_a$
and $[\varphi ,x_i]=0 \quad \forall i>1$.  Now, for every $i>1$, since
$[\varphi ,x_1x_i] = q_{1i}[\varphi ,x_ix_1]$, we have
$\alpha \delta_J \sigma_a x_i = q_{1i}\alpha  x_i \delta_J \sigma_a$. That is,
there exist $ c_1,c_2\in \K^*$ such that
\[
(\beta (a,e_i) -q_{1i}) \alpha c_1\rho_{x_i} \delta_J \sigma_{a+e_i}
+c_2 \alpha j_i \delta_{J-e_i}\sigma_{a+e_i}=0.
\]
By linear independence of $\rho_{x_i} \delta_J \sigma_{a+e_i}$ and
$\delta_{J-e_i}\sigma_{a+e_i}$, if $\alpha \neq 0$
then $\beta (a,e_i) =q_{1,i}$. That is, $a-e_1 \in \Z e_i$ and $j_i =0$ for every $i>1$.
Thus $a =e_1$ and  $j_i =0$ for every $i>1$.
\end{proof}
\begin{lemma} \label{Lm2:Dq-qtorus}
Suppose $\varphi \in D_q(R)$ is such that $[\varphi ,x_i] =0$ and for some $j\neq i$
$[\varphi ,x_j] = \sum_{I,J,a}\alpha_{I,J,a} \rho_I \delta_J \sigma_a \neq 0$.
Then there is an $I_0$ such that whenever $\alpha_{I_0, J,a}\neq 0$, we have
$a-e_j \in \Z e_i$ and $j_i =0$.
\end{lemma}
\begin{proof}
To avoid confusion with several indices, we show that if
$[\varphi ,x_1]=0$ and $[\varphi ,x_2]=\sum_{I,J,a}\alpha_{I,J,a} \rho_I \delta_J \sigma_a \neq 0$
then there is an $I_0$ such that whenever $\alpha_{I_0, J,a}\neq 0$, we have
$a-e_2 \in \Z e_1$ and $j_1 =0$. The general case can be inferred.

Since $[\varphi ,x_1x_2] = q_{12}[\varphi ,x_2x_1]$, we have
\begin{align*}
x_1\sum_{I,J,a}\alpha_{I,J,a} \rho_I \delta_J \sigma_a
            =& q_{12}\left( \sum_{I,J,a}\alpha_{I,J,a} \rho_I \delta_J \sigma_a \right) x_1\\
	=& x_1 \sum_{I,J,a}\beta (a, e_1) q_{12} \alpha_{I,J,a} \rho_I \delta_J \sigma_a\\
           &+ \sum_{I,J,a}\beta (a,e_1) q_{12}\alpha_{I,J,a} c_J j_1\rho_I \delta_{J-e_1} \sigma_{a+e_1},
\end{align*}
for some $c_J \in \K^*$.
Therefore, by setting $q_{a1}:= \beta (a,e_1)$ we have
\begin{align*}
x_1\sum_{I,J,a}(1-q_{a1}q_{12})\alpha_{I,J,a} \rho_I \delta_J \sigma_a
   &= \sum_{I,J,a}q_{a1} q_{12}\alpha_{I,J,a} c_J j_1\rho_I \delta_{J-e_1} \sigma_{a+e_1}.
	\end{align*}
Using the fact that $\lambda_{x_1} = \rho_{x_1} \sigma_{e_1}$ there exist $d_J \in \K^*$ such that
\begin{equation}\label{eqn:Lemma2forisothm}
\sum_{I,J,a}(1-q_{a1}q_{12})\alpha_{I,J,a}d_J \rho_{I+e_1} \delta_J \sigma_{a+e_1}
 =   \sum_{I,J,a}q_{a1} q_{12}\alpha_{I,J,a} c_J j_1\rho_I \delta_{J-e_1} \sigma_{a+e_1}.
\end{equation}
Let $I_0$ be the greatest index with respect to lexicographical ordering such that
$\alpha_{I_0, J,a} \neq 0$ for some $J,a$. Then, the equation \eqref{eqn:Lemma2forisothm}
implies that $q_{a1} = q_{21}$ whenever $\alpha_{I_0, J,a} \neq 0$; that is,
whenever $\alpha_{I_0, J,a} \neq 0$ we have $q - e_2 \in \Z e_1$.

Moreover, we have
\begin{align*}
&\rho_{I_0} \left( \sum_{J,a}(1-q_{a1}q_{12})\alpha_{I_0-e_1,J,a}d_J \delta_J \sigma_{a+e_1} \right) \\
&+
\sum_{I<I_0, I\neq I_0 -e_1}
\rho_I  \left( \sum_{J,a}(1-q_{a1}q_{12})\alpha_{I_-e_1,J,a}d_J \delta_J \sigma_{a+e_1} \right) \\
=&
\rho_{I_0} \left( \sum_{J,a}q_{a1} q_{12}\alpha_{I_0,J,a} c_J j_1\delta_{J-e_1} \sigma_{a+e_1} \right) \\
&+
\sum_{I<I_0}
\rho_{I} \left( \sum_{J,a}q_{a1} q_{12}\alpha_{I,J,a} c_J j_1\delta_{J-e_1} \sigma_{a+e_1} \right)
\end{align*}
Equating the coefficients of $\rho_{I_0}$ we get
\begin{equation}\label{equation2}
\sum_{J,a}(1-q_{a1}q_{12})\alpha_{I_0-e_1,J,a}d_J \delta_J \sigma_{a+e_1} =
\sum_{J,a}q_{a1} q_{12}\alpha_{I_0,J,a} c_J j_1\delta_{J-e_1} \sigma_{a+e_1}.
\end{equation}
Recall that $\alpha_{I_0,J,a}\neq 0$ implies that $a-e_2 \in \Z e_1$. Hence,
every $a\in \Z^n$ in the left hand side of the above equation \eqref{equation2}
is  such that $a-e_2 \in \Z e_1$.
That is, $q_{a1}=q_{21}$. In other words, the left hand side of equation \eqref{equation2} is zero.
The right hand side of equation \eqref{equation2} is therefore zero. That is, whenever
$\alpha_{I_0,J,a}\neq 0$, we get $j_1 =0$.

\end{proof}
\begin{corollary}\label{Cor1:Dq-qtorus}
If homogeneous $\varphi \in D_q(R)$ is such that $[\varphi ,x_i]=[\varphi ,x_j]=0$ for some $i\neq j$
and $[\varphi ,x_k]\in f(\Delta \# \Z^n)$ for all $k\neq i,j$. Then $\varphi \in f(\Delta \# \Z^n)$.
\end{corollary}
\begin{proof}
Without loss of generality assume that $[\varphi ,x_1]=[\varphi ,x_2]=0$
and $[\varphi ,x_3] \in f(\Delta \# \Z^n)$.
Following the proof of Lemma \ref{Lm2:Dq-qtorus} the multi-index $I_0$ of the highest order in lexicographical
ordering and multi-index $J_0=(0,0,j_3^0,\ldots ,j_n^0)$  are such that
\[
[\varphi ,x_3] =\rho_{I_0} \delta_{J_0}\left( \sum_{m\in \Z }\alpha_{m}  \sigma_{e_3+me_1} \right) +
                     \sum_{I< I_0} \rho_{I}  \delta_J \left( \sum_{a\in \Z^n}\alpha_{a } \sigma_{a} \right)
\]
and
\[
[\varphi ,x_3] =\rho_{I_0} \delta_{J_0}\left( \sum_{s\in \Z }\gamma_{s}  \sigma_{e_3+se_2} \right) +
                     \sum_{I< I_0} \rho_{I}  \delta_J \left( \sum_{a\in \Z^n}\gamma_{a } \sigma_{a} \right)
\]
for some $\alpha_m, \alpha_a \gamma_m, \gamma_a \in \K$.
Therefore, we have
\[
[\varphi ,x_3] = \alpha \rho_{I_0} \delta_{J_0}\sigma_{e_3}+
                     \sum_{I< I_0} \rho_{I}  \delta_J \left( \sum_{a\in \Z^n}\gamma_{a } \sigma_{a} \right)
\]
for some $\alpha \in \K$.
Since $J_0=(0,0,j_3^0,\ldots ,j_n^0)$, we have
\[
[ \varphi - \alpha \rho_{I_0}\delta_{J_0+e_3} ,x_1] = [ \varphi - \alpha \rho_{I_0}\delta_{J_0+e_3} ,x_2]=0
\]
and
$[ \varphi - \alpha \rho_{I_0}\delta_{J_0+e_3} ,x_3] =   \sum_{I< I_0} \rho_{I}  \delta_J \left( \sum_{a\in \Z^n}\gamma_{a } \sigma_{a} \right)$.
Proceeding similarly, there is a $\psi \in f(\Delta \# \Z^n)$ such that
\[
[\varphi - \psi ,x_1]=[\varphi -\psi ,x_2]= [\varphi -\psi ,x_3] =0.
\]
Continuing by induction, we prove the corollary.
\end{proof}
\begin{lemma}\label{Lm3:Dq-qtorus}
Suppose homogeneous $\varphi \in D_q(R)$ is such that $[\varphi ,x_i] =0$ and  for some
$j\neq i$,
$[\varphi ,x_j] = \sum_{I,J,a}\alpha_{I,J,a} \rho_I \delta_J \sigma_a \neq 0$.
Then there exists a  $\psi \in f (\Delta \# \Z^n)$ such
that $[\varphi  - \psi ,x_i ] = [\varphi - \psi ,x_j ] = 0$.
\end{lemma}
\begin{proof}
Without loss of generality, we may assume that  $[\varphi ,x_1] =0$ and
$[\varphi ,x_2] = \sum_{I,J,a}\alpha_{I,J,a} \rho_I \delta_J \sigma_a \neq 0$.
Following the proof of Lemma \ref{Lm2:Dq-qtorus} the multi-index $I_0$ of the highest order in lexicographical
ordering is such that
$\alpha_{I_0,J,a}\neq 0$ implies that $j_1=0$ and $a=e_2 +n_a e_1$ for some $n_a \in \Z$ dependent
on $a$. Thus, there exist $\alpha_m , \alpha_a \in \K^*$ such that
\[
[\varphi ,x_2] =\rho_{I_0} \delta_{J_0}\left( \sum_{m\in \Z }\alpha_{m}  \sigma_{e_2+me_1} \right) +
                     \sum_{I< I_0} \rho_{I}  \delta_J \left( \sum_{a\in \Z^n}\alpha_{a } \sigma_{a} \right).
\]
Suppose degree of $\varphi$ is $T=(t_1,t_2,\ldots, t_n)$ and $I_0 =(i_1^0, i_2^0, \ldots ,i_n^0)$.
Then $J_0 = I_0 -e_2 - T$.
Since $\delta_1$ is not a factor of $\delta_{J_0}$, we have $t_1 = i_1^0 \geq 0$.
Thus, for any $I<I_0$ for which $\alpha_a \neq 0$ in the above sum, we have $I=(i_1,i_2,\ldots , i_n)$
with $i_1 = i_1^0 \geq 0$ (and $\del_1$ is not a factor of any of the corresponding $\del_J$). That is, $\rho_{x_1}^{t_1}$ is a factor of $[\varphi ,x_2]$.  Similarly,
$\rho_{x_1}^{t_1}$ is a factor of $[\varphi ,x_i]$ for every $i\geq 2$. If we prove the lemma for homogeneous
$\psi$ with degree $(0,t_2, \ldots ,t_n)$ then the general case is a multiple of $\rho_{x_1}^{t_1}$. So without
loss of generality assume that $t_1=i_1^0=0$ and that $\del_1$ is not a factor of any of $J_0$ and $J$
corresponding to those $I$ such that $\alpha_a \neq 0$ in the above equation and any
commutator equation [$\varphi ,x_i]$ for $i\geq 2$.

We now show that there exists $\psi \in f( \Delta \# \Z^n)$ such that
$[\varphi - \psi ,x_1] = [\varphi - \psi, x_2]=0$. It is enough to show that there exists $\psi \in f( \Delta \# \Z^n)$ such that
$[\varphi - \psi ,x_1] = 0$ and
\[
[\varphi - \psi, x_2]= \sum_{I< I_0} \rho_{I}  \delta_J \left( \sum_{a\in \Z^n}\gamma_{a } \sigma_{a} \right)
\]
for some $\gamma_a \in \K$.
In other words, we need to show the existence of $\psi \in f( \Delta \# \Z^n)$ such that
$[\psi ,x_1]=0$ and
\[
[\psi ,x_2] = \rho_{I_0} \delta_{J_0}\left( \sum_{m\in \Z }\alpha_{m}  \sigma_{e_2+me_1} \right)
 + \sum_{I< I_0} \rho_{I}  \delta_J \left( \sum_{a\in \Z^n}r_{a } \sigma_{a} \right)
\]
for some $r\in \K$. Note that $[\rho_{I_0}\delta_{J_0+e_2} , x_1] =0$ and
$[\rho_{I_0}\delta_{J_0+e_2}, x_2] = \sigma_{e_2}$. Further, for any $m\neq 0$, if
$I_0 = (0,i_2^0,\ldots ,i_n^0)$
is such that $i_2^0 >0$ then there is $c\in \K^*$ such that
$[c\rho_{I_0 - e_2} \delta_{J_0} \sigma_{me_1} ,x_1]=0$ and
\[
[c\rho_{I_0 - e_2} \delta_{J_0} \sigma_{me_1} ,x_2] =
 \rho_{I_0} \delta_{J_0} \sigma_{e_2+me_1}
 + \sum_{I< I_0} \rho_{I}  \delta_J \left( \sum_{a\in \Z^n}s_{a } \sigma_{a} \right)
\]
for some $s_a\in \K$.
It remains to show that if
\[
[\varphi ,x_2] =\rho_{I_0} \delta_{J_0}\left( \sum_{m\in \Z }\alpha_{m}  \sigma_{e_2+me_1} \right) +
                     \sum_{I< I_0} \rho_{I}  \delta_J \left( \sum_{a\in \Z^n}\alpha_{a } \sigma_{a} \right)
\]
then  $i_2^0 =0$, and $m\neq 0$ implies that $\alpha_m = 0$.

Suppose $i_2^0 = 0$ and $\alpha_{m_0} \neq 0$ for some $m_0\neq 0$.
In view of Lemma
\ref{Lm1:Dq-qtorus} we assume that there is an $i>2$ such that $[\varphi ,x_i]\neq 0$.
Without loss of generality, let
\[
[\varphi ,x_3] =\rho_{P_0} \delta_{Q_0}\left( \sum_{s\in \Z }\zeta_{s}  \sigma_{e_3+se_1} \right) +
                     \sum_{P< P_0} \rho_{P}  \delta_Q \left( \sum_{a\in \Z^n}\zeta_{a } \sigma_{a} \right)
\]
with $\del_1$ not a factor of $\delta_{Q_0}$ and $P_0 = (0,p_2^0,\ldots ,p_n^0)$.
Now consider the identity $[\varphi , x_2x_3] = q_{23}[\varphi ,x_3x_2]$.
That is,
\[
[\varphi ,x_2]x_3 - q_{23}x_3[\varphi ,x_2] = q_{23}[\varphi ,x_3]x_2 - x_2[\varphi ,x_3].
\]
That is, there are scalars $c_1, c_2\in \K^*$ such that
\begin{align*}
(q_{13}^{m_0} -1)c_1 \rho_{I_0+e_3}\delta_{J_0} \sigma_{e_2+e_3+m_0e_1} &+ \textit{ rest }\\
= (q_{13}^{m_0} -1)c_2 &\rho_{P_0+e_2}\delta_{Q_0} \sigma_{e_2+e_3+m_0e_1} + \textit{ rest }.
\end{align*}
In other words, $I_0 +e_3 = P_0 +e_2$ which would imply that $i_2^0 \neq 0$ unless
$m_0 =0$. This contradicts our assumption made in the beginning of this paragraph.
We have thus proved the lemma.
\end{proof}

\begin{theorem}\label{Th:Dq-qtorus}
The map $f: \Delta \# \Z^n \to D_q(R)$ is an isomorphism of filtered
algebras.
\end{theorem}
\begin{proof}
By Theorem~\ref{DZnsimple}, $\Delta \# \Z^n$ is simple, and hence $f$ is injective.
It remains to show that $f$ is surjective and that $f$ is a map of filtered algebras.

To show that $f$ is surjective, it suffices to show that if a homogeneous
$\varphi \in \grHom(R,R)$ is such that $[\varphi , x_i] \in f(\Delta \# \Z^n)$ for every
$i\leq n$, then $\varphi \in f(\Delta \# \Z^n)$.

Suppose that $[\varphi , x_i]=0$ for every $i\leq n$. Then $\varphi = \rho_{\varphi (1)} \in
f(\Delta \# \Z^n)$.  Moreover, in view of Lemmas \ref{Lm1:Dq-qtorus} - \ref{Lm3:Dq-qtorus}
and Corollary \ref{Cor1:Dq-qtorus}
assume that $[\varphi ,x_i]\neq 0$ for every $i$.

Let $[\varphi , x_1] = \sum_{I,J,a} \alpha_{I,J,a} \rho_I \delta_J \sigma_a \in f(\Delta \# \Z^n)$.
Note that if $a\in \Z e_1$, then by formula (\ref{commutator})
\[
[\rho_I \delta_{J+e_1} \sigma_{a-e_1} , x_1] = c \rho_I \delta_J \sigma_a \quad \text{for some }
	c\in \K^*.
\]
Similarly, by the formula \ref{commutator}, for any $a\notin \Z e_1$ and $\rho_{x_1}$ a factor of
$\rho_I$,
\[
[\rho_{I-e_1} \delta_J \sigma_a ,x_1] = c_1 \rho_{I-e_1}\delta_{J-e_1} \sigma_{a+e_1} +
c_2\rho_I \delta_J \sigma_a \quad \text{for some }c_1, c_2\in \K^*.
\]
Thus, we may replace $\varphi$ by $\varphi - \psi$ for some $\psi \in f(\Delta \# \Z^n)$ (and use
induction on $m$ where $[\varphi , x_1] \in f(\Delta_m \# \Z^n)$) so that
$[\varphi , x_1] = \sum_{I,J,a} \alpha_{I,J,a} \rho_I \delta_J \sigma_a $,
where $\alpha_{I,J,a}\neq 0$ implies that $a\notin \Z e_1$ and $\rho_{x_1}$ is not a factor
of $\rho_I$. We claim that this will imply that $[\varphi ,x_1] =0$.

Consider $[\varphi , x_1x_i ] =q_{1i}[\varphi ,x_i x_1]$.  That is,
\begin{equation}\label{eqn:comm}
[\varphi, x_1]x_i - q_{1i}x_i[\varphi ,x_1] = q_{1i}[\varphi, x_i]x_1  -  x_1[\varphi ,x_i] .
\end{equation}
Note that $\rho_{x_1}$ is not a factor in any of the terms which appear in the left hand side.
Fix an $i>1$.  Let
$[\varphi ,x_i] = \sum_{M,N,b} \alpha_{M,N,b} \rho_M \delta_N \sigma_b$. Then,
there exist $c_{M,N,b} \in \K^*$ such that
\begin{align*}
q_{1i}[\varphi, x_i]x_1  -  x_1[\varphi ,x_i] &= x_1 \sum_{M,N,b} \left( q_{1i}\beta (b,1) -1 \right)
\alpha_{M,N,b} \rho_M \delta_N \sigma_b \\
 &+
\sum_{M,N,b}n_1 \alpha_{M,N,b}c_{M,N,b} \rho_M \delta_{N-e_1}\sigma_{b+e_1}.
\end{align*}
Since $\rho_{x_1}$ (and hence $x_1$) is not a factor of any term in the left
hand side of \eqref{eqn:comm}, we have $q_{1i}\beta (b,1) =1$ for every $b$ whenever
$\alpha_{M,N,b}\neq 0$. That is, $b \in e_i + \Z e_1$. This further implies that $b+e_1 \in e_i + \Z e_1$.
Since $\alpha_{I,J,a}\neq 0$ implies that $a\notin \Z e_1$, again using \eqref{eqn:comm}
we conclude that $n_1 = 0$ whenever
$ \alpha_{M,N,b}\neq 0$. Thus, $q_{1i}[\varphi, x_i]x_1  -  x_1[\varphi ,x_i] =0$.
Hence by \eqref{eqn:comm}, for every $i>1$, we have $[\varphi, x_1]x_i = q_{1i}x_i[\varphi ,x_1]$.

That is, $[[\varphi,x_1]  \sigma_{e_1}^{-1} ,x_1] \in f(\Delta \# \Z^n)$ and
$[[\varphi,x_1]  \sigma_{e_1}^{-1} ,x_i]=0$
for all $i>1$. By the Lemma \ref{Lm1:Dq-qtorus} we have
$[[\varphi,x_1] \sigma_{e_1}^{-1} ,x_1]  = [\varphi ,x_1]  \sigma_{e_1}^{-1} = T \sigma_{e_1}$
for some $T$
in the subalgebra of $f(\Delta \# \Z^n)$ generated by the set
$\{ \rho_{x_i} \}_{1\leq i \leq n} \cup \{ \delta_1 \}$.  That is, $ [\varphi ,x_1]= T \sigma_{2e_1}$.
This contradicts the assumption on $\varphi$ unless $T=0$.  So, we have
$[\varphi  ,x_1]=0$. Now we appeal to Lemmas \ref{Lm2:Dq-qtorus} and \ref{Lm3:Dq-qtorus}
and Corollary \ref{Cor1:Dq-qtorus} to claim the conclusion of
the theorem.

\end{proof}

\begin{corollary}\label{Dqnspace}
The algebra
$D_q(R)$ is a left and right Noetherian, simple
domain, with GK-dimension $3n$.
\end{corollary}

\subsubsection{Quantum tori in two variables}\label{qtori}
Let
\begin{align*}
&A_1 = D_q (\K \langle x,x^{-1},y\rangle  /(xy=qyx)),
\;
A_2 = D_q (\K \langle x,y,y^{-1}\rangle  /(xy=qyx)),\\
&\textit{and }
A_3 = D_q (\K \langle x,x^{-1},y,y^{-1}\rangle  /(xy=qyx)).
\end{align*}
Thus
\begin{align*}
&A_1 \textit{ is generated by }
\{ \lambda_x, \lambda_{x^{-1}}, \rho_y, \del_x^a, \del_y^a \mid
    a\in \{-1, 0,1\} \};\\
&A_2 \textit{ is  generated by }
\{ \lambda_x, \rho_y, \rho_{y^{-1}}, \del_x^a, \del_y^a \mid
    a\in \{-1, 0,1\} \};\\
&A_3 \textit{ is  generated by }
\{ \lambda_x, \lambda_{x^{-1}},\rho_y, \rho_{y^{-1}}, \del_x^a, \del_y^a \mid
    a\in \{-1, 0,1\} \}.
\end{align*}

We first note the following:
\begin{corollary}\label{datt2}
There exists an isomorphism of the algebra $A_3$ and its
opposite algebra, $A_3^o$, such that the isomorphism restricts to
an isomorphism of $D_q(R)$ (respectively, $A_1$ and $A_2$) and its opposite, $D_q(R)^o$
(respectively, $A_1^o$ and $A_2^o$) where $R= \K<x,y>/(xy=qyx)$ is the coordinate algebra
of the quantum plane.
\end{corollary}
\begin{proof}

By Section \ref{section:quantum-plane} and Remark
\ref{remark:localization},
 $A_3 \cong D_q(\K [t,t^{-1}]) \otimes D_q(\K [t,t^{-1}])$ as
filtered algebras. Now use Corollary \ref{datt} to complete the proof.
\end{proof}

Let $R$ be the quantum torus $\K[x^{\pm 1}, y^{\pm 1}:xy=qyx]$, where $q$ is not a root of unity.
$D_q(R)$ is generated by the set $\{\lambda_x, \lambda_{x^{-1}}, \rho_y, \rho_{y^{-1}}, \del_x^a, \del_y^a\}$ for $a\in \{-1, 0,1\}\}.$
Following the same approach as for $\K[x^{\pm 1}]$, we aim to include the following operators, which act as automorphisms on $R$, among the generators (using notation from Remark \ref{different-betas} and
in Section \ref{section:quantum-plane}):
\begin{eqnarray*}
\sigma_{x}:x^iy^j\mapsto q^ix^iy^j;&\sigma_{x}^{-1}:x^iy^j\mapsto q^{-i}x^iy^j;\\
\sigma_{y}:x^iy^j\mapsto q^jx^iy^j;&\sigma_{y}^{-1}:x^iy^j\mapsto q^{-j}x^iy^j.
\end{eqnarray*}
Using the identities
\begin{eqnarray*}
\rho_y&=&\lambda_y \sigma_{x},\\
\rho_{y^{-1}}&=&\sigma_{x}^{-1}\lambda_{y^{-1}},\\
\del_x^1&=&(q-1)^{-1}\lambda_{x^{-1}}(\sigma_{x}-1),\\
\del_x^{-1}&=&(q^{-1}-1)^{-1}\lambda_{x^{-1}}(\sigma_{x}^{-1}-1),\\
\del_y^1&=&(q-1)^{-1}\rho_{y^{-1}}(\sigma_{y}-1),\\
\del_y^{-1}&=&(q^{-1}-1)^{-1}\rho_{y^{-1}}(\sigma_{y}^{-1}-1),
\end{eqnarray*}
we can work with the alternative generating set
\[
\{\lambda_x, \lambda_{x^{-1}}, \lambda_y, \lambda_{y^{-1}}, \del_x, (\del_y \sigma_x),
\sigma_{x}^{\pm 1},\sigma_{y}^{\pm 1}\},
\]
where $\del_x=\del_x^{\beta^0}$ and
$\del_y=\del_y^{\beta^0}$ as described in Section \ref{section:quantum-plane}. As this only involves left multiplications, we shall, for $z\in \{x^{\pm 1},y^{\pm 1}\}$, write $z$ for $\lambda_z$, and thereby view $R$ as a subalgebra of $D_q(R)$ acting by left multiplication.

Let $R_1$ be the subalgebra generated by $x, y, \del_x$ and $(\del_y \sigma_x)$ and let is $A^{\Lambda}_2$ be the quantized Weyl algebra  as in proof of Theorem \ref{Deltasimple}
with  $\Lambda=
	\begin{pmatrix}
	1 &q\\
	q^{-1} &1
	\end{pmatrix}$. There is a surjective homomorphism $\theta: A^{\Lambda}_2\twoheadrightarrow  R_1$,  with  $u_1\mapsto \del_x, u_2\mapsto (\del_y \sigma_x),v_1\mapsto x$ and $v_2\mapsto y$. By  \cite[Theorem 6.5]{JW}, $A^{\Lambda}_2$ is simple so  $\theta$ is an isomorphism. Thus $R_1$ is simple and right and left Noetherian. It follows that the localization $R_2$ of $R_1$ obtained by inverting $x$ and $y$ is also simple and right and left Noetherian.

Let $\alpha$ be the $\K$-automorphism of $R_2$ such that $\alpha(x)=qx$, $\alpha(\del_x)=q^{-1}\del_x$,
$\alpha(y)=y$ and $\alpha(\del_y\sigma_x)=(\del_y\sigma_x)$. Clearly  $\alpha$ is not inner on $R_2$ where the units have the form $\kappa x^iy^j$, $0\neq\kappa\in\K, i,j\geq 0$.
By \cite[Theorem 1.8.5]{McCR}, the skew Laurent polynomial ring $R_2[X^{\pm 1};\alpha]$ is simple and it follows, since $\sigma_{x}x=qx\sigma_{x}$,
$\sigma_{x}\del_x=q^{-1}\del_x\sigma_{x}$, $\sigma_{x}y=y\sigma_{x}$ and $\sigma_{x}(\del_y\sigma_x)=(\del_y\sigma_x)\sigma_{x}$ that the $\K$-algebra extension $R_3$ of $R_2$ generated by
$\sigma_{x}^{\pm 1}$ is isomorphic to  $R_2[X^{\pm 1};\alpha]$. Finally, let $\gamma$ be the $\K$-automorphism of $R_3$ such that $\gamma(x)=x$, $\gamma(\del_x)=\del_x$,
$\gamma(y)=q^{-1}y$, $\gamma(\del_y\sigma_x)=q\del_y\sigma_x$ and $\gamma(\sigma_{x})=\sigma_{x}$. By similar arguments to those above, $D_q(R)\simeq R_3[Y^{\pm 1};\gamma]$ and is simple and right and left Noetherian. We have thus proved the following:
\begin{theorem}\label{theorem:A1,A2,A3}
The algebras $A_1, A_2$ and $A_3$ are simple domains, which are left and right Noetherian.
\end{theorem}
\subsubsection{\textbf{The algebra $D_q(R_{s,n})$, $n\geq 3$}}

By Proposition \ref{prop:localization}, the algebra $D_q(R_{s,n})$
is the localization of the algebra $D_q(R)$ with respect to the
Ore set $S$  generated by $\{ x_1, \ldots, x_s \}$. In other words,
$D_q(R_{r,n}) \cong S^{-1}\Delta \# \Z^n$. As in Corollary~\ref{Dqnspace}, we see that
$D_q(R_{s,n})$ is a Noetherian, simple domain
of GK-dimension $3n$.

The contruction of $D_q(R_{s,n})$ for $0\leq s \leq n$
was with respect to the matrix of parameters
$(q_{ij})$.  Notice that the
algebra of quantum differential operators
correspond to the transpose of the matrix $(q_{ij})$ is isomorphic to the algebra
$(D_q(R_{s,n}))^o$.

\begin{corollary}\label{tori-sev-opp}
There is an isomorphism $\Phi : D_q(R_{n,n}) \to (D_q(R_{n,n}))^o$
which restricts to isomorphism from $D_q(R_{s,n})$ to its opposite for
every $s, 0\leq s\leq n$.
\end{corollary}
\begin{proof}
The required map $\Phi  : D_q(R_{n,n}) \to (D_q(R_{n,n}))^o$ is given by
\[
\Phi ( \del_i) = - (\sigma_{-e_i}\del_i)^o, \quad
\Phi ( \rho_{x_i}) =  (\sigma_{e_i}\rho_{x_i})^o, \quad
\Phi ( \sigma_{\gamma}) =  (\sigma_{\gamma})^o
\]
for $1\leq i \leq n, \gamma \in \Z^n$.
\end{proof}
\section{The quantum exterior algebra}\label{ext}
Here we consider the multiparameter quantum exterior algebra, $R$, in $n$ variables,
$\xi_1,\xi_2, \cdots ,\xi_n$ with parameters $p_{ij}$ such that
$p_{ji}=\dfrac{1}{p_{ij}}$ and $p_{ii}=-1$. Then
\[
R = \K\langle  \xi_1,\xi_2,\cdots ,\xi_n\rangle  / (\xi_i^2 =0, \xi_i \xi_j = -p_{ij}\xi_j\xi_i)
\]
Then $R$ is a finite dimensional $\K$-algebra, graded by
$\mathbb{Z}^n$ where degree of $\xi_i$ is $e_i$, the standard
basis element, for each $i=1,2,\cdots ,n$.
Set $\beta :\Z^n \times \Z^n \to \K^*$ be given by
$\beta (e_i, e_j) = -p_{ij}$.
As usual, denote by $\sigma_{e_i} \in D^0_q(R)$ the grading automorphism
on $R$ defined by $\sigma_{e_i} (r) = \beta (e_i, d_r) r$, where
$d_r$ denotes the degree of $r$.
\begin{definition}
For each $i\leq n$, denote by $\del_i^{-e_i}$ the left
$\sigma_{-e_i}$-derivation on $R$ defined by
\[
\del_i^{-e_i}(\xi_i) = 1\quad \del_i^{-e_i}(\xi_j) = 0, \quad \textit{for }
j\neq i, \textit{  and }
\del_i^{-e_i}(1)=0.
\]
\end{definition}
Note that $\del_i^{-e_i} (\xi_k \xi_l + p_{kl} \xi_l \xi_k)=0$
and $\del_i^{-e_i}(\xi_k^2) =0$ for any $i,k,l \leq n$.
Each $\del_i^{-e_i} \in D_q^1(R)$. The following theorem is
similar to the result in  of \cite[Section 4.1]{I}.
\begin{theorem}
The algebra $D_q(R)=\Hom _{\K}(R,R)$.
\end{theorem}
\begin{proof}
Let $\varphi = \del_n^{-e_n}\del_{n-1}^{-e_{n-1}}\cdots \del_1^{-e_1}
\in D^n_q(R)$.  Then $\varphi (\xi_1\xi_2\cdots \xi_n)=1$
and $\varphi (r) =0$ for any $r\in R$ of lower degree.

A basis for $R$ is given by the set
$B= \{ \xi_{i_1}\xi_{i_2}\cdots \xi_{i_n} \}_{1\leq i_1<i_2<\cdots <i_n}
\cup \{ 1 \}$. For any $b_1, b_2\in B$, there exists $b\in B$ and
$\alpha \in \K$ such that $\alpha b b_1 = \xi_1 \xi_2\cdots \xi_n$
and $\alpha b b_j =0$ for $b_j\neq b_1, b_j \in B$.
Hence, $\alpha \lambda_{b_2} \varphi \lambda_{b} (b_1) =b_2$
and $\alpha \lambda_{b_2} \varphi \lambda_{b} (b_j)=0$ for any
$b_j \neq b_1, b_j \in B$. Hence $\Hom _{\K}(R,R) \subset D_q(R)$.
\end{proof}


\bibliographystyle{amsplain}

\end{document}